\documentclass[11pt]{amsart}
\usepackage{amssymb,amsmath,amsthm}
\usepackage[hmargin=3cm,vmargin=3.5cm]{geometry}

\usepackage{graphicx}
\usepackage{amscd}
\usepackage[all, knot]{xy}
\usepackage{epsfig}
\usepackage{psfrag}
\usepackage{amssymb,amsmath,amsthm}
\usepackage{amsfonts}
\usepackage{amssymb}

\newtheorem{theorem}{Theorem}[section]
\newtheorem{lemma}[theorem]{Lemma}
\newtheorem{remark}[theorem]{Remark}
\newtheorem{definition}[theorem]{Definition}

\newtheorem{proposition}[theorem]{Proposition}
\newtheorem{prop}[theorem]{Proposition}
\newtheorem{cor}[theorem]{Corollary}

\newcommand\Z{{\mathbb{Z}}}

\newcommand{\oplusop}[1]{{\mathop{\oplus}\limits_{#1}}}
\newcommand{\oplusoop}[2]{{\mathop{\oplus}\limits_{#1}^{#2}}}

\def\sbinom#1#2{\left( \hspace{-0.06in}\begin{array}{c} #1 \\ #2
\end{array}\hspace{-0.06in} \right)}
\def\fieldk{\mathbf{k}}
\def\dmod{\mathrm{-mod}}
\def\pmod{\mathrm{-pmod}}
\def\define{\stackrel{\mathrm{def}}{=}}
\def\lra{\longrightarrow}
\def\ra{\rightarrow}
\def\Hom{\mathrm{Hom}}

\def\mc{\mathcal}

\title{Categorifications of the polynomial ring}

\author{Mikhail Khovanov, Radmila Sazdanovi\'c\\ \, \\
January 1, 2011}

\begin{document}

\maketitle

\textbf{Abstract:} We develop a diagrammatic categorification of the polynomial ring $\Z[x]$. Our categorification satisfies a version of Bernstein-Gelfand-Gelfand reciprocity property with the indecomposable projective modules corresponding to $x^n$ and standard modules to $(x-1)^n$ in the Grothendieck ring.
\tableofcontents


\section{Introduction}

Inspired by the general idea of categorification, introduced by L.~Crane and I.~Frenkel,
we construct a categorification of the polynomial ring $\Z[x]$, more precisely of polynomials $(x-1)^n$ that can be generalized
to orthogonal one-variable polynomials, including Chebyshev polynomials of the
second kind and the Hermite polynomials \cite{KS}.

In this paper, we interpret the ring $\Z[x]$ as the Grothendieck ring of a suitable additive monoidal category $A^-\pmod$ of
(finitely generated) projective modules over an idempotented geometrically defined ring $A^-.$ Monomials $x^n$
 become indecomposable projective modules $P_n$, while polynomials $(x-1)^m$ turn into so-called standard modules
  $M_m.$ Ring  $A^-$ has one more distinguished family of modules - simple modules $L_n.$ A remarkable feature of these
  three collections of modules is the Bernstein--Gelfand--Gelfand (or BGG) reciprocity property \cite{BGG}. Projective modules $P_n$
   have a filtration by standard modules $M_m$, for $m \leq n$, and the multiplicities satisfy the relation:
   $$[P_n:M_m]=[M_m:L_n].$$

Original examples of algebras and modules with this property are due to J.~Bernstein, I.~Gelfand, and S.~Gelfand  and come up
in infinite-dimensional representation theory of simple Lie algebras. The algebra $A^-$ has a purely
topological-geometric definition, yet satisfies the BGG property. Moreover, the standard modules $M_n$ have a clear
geometric interpretation. An additional sophistication  appears due to non-unitality of algebras $A^-.$ Instead,
they contain an infinite collection of idempotents $1_n,$ $n \geq 0,$ serving as a substitute for the unit
element $1.$ Projectives $P_n$ and standard modules $M_n$ are infinite-dimensional, and the multiplicity $[M_m:L_n]$
 should be understood in the generalized sense, as dim$(1_nM_m).$ We hope that our approach will lead to
 geometric interpretation of the BGG reciprocity in many other cases, including the ones
 considered by  J.~Bernstein, I.~Gelfand, and S.~Gelfand.
 In the sequel \cite{KS} we will generalize this constructions to categorify the Hermite and Chebyshev polynomials.



\section{The algebra of slarcs and what it categorifies}

Denote by $_m B^-_n$ the set of isotopy classes of planar diagrams
(see Figure  ~\ref{dang1}) which connect $k$ out  of $m$ points on
the line $x=0$ to $k$ out of $n$ points on the line $x=1$ by $k$
arcs called larcs (long arcs), $k\le \min(n,m).$ The remaining
$m-k$ left and $n-k$ right points extend to \emph{short arcs} or
\emph{sarcs}, with one endpoint on either line $x=0$ or $x=1$
and the other in the interior of the strip $0<x<1.$ We require
that the projection of the resulting 1-manifold onto the $x$-axis
has no critical points. The number of larcs $k$ is called the
\emph{width} of the diagram. Let $_m B^-_n(k)$ and $_mB^-_n(\leq k)$
denote the subsets of diagrams in  $_m B^-_n$ of
width $k$ and less than or equal to $k$, respectively.

\begin{figure} \scalebox{1.1}{\includegraphics{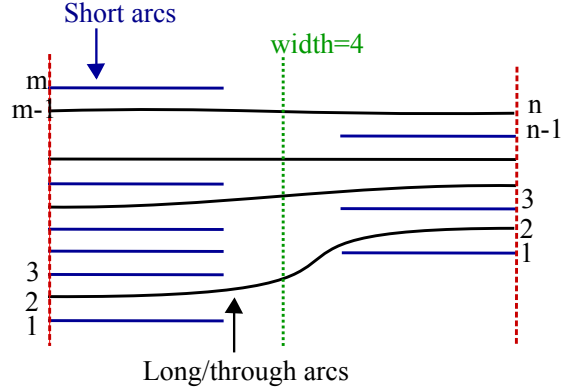}}
\caption{A diagram in $_m B^-_n$.}
\label{dang1}
\end{figure}

The set $_m B^-_n$ has cardinality
$$\sum_{k=0}^{\min(n,m)} \sbinom{n}{k}\sbinom{m}{k} = \sbinom{n+m}{n}.$$

Let
$$ B^- \define \displaystyle{\bigsqcup_{n,m\ge 0}  \hspace{0.02in} _m B^-_n} $$
and
$$ B^-_n \define \displaystyle{\bigsqcup_{n\ge 0}  \hspace{0.02in} _m B^-_n}. $$

Given a field $\fieldk$, form $\fieldk$-algebra $A^-$ as a vector space with the basis $B^-$ and
the multiplication generated by the concatenation of elements of
$B^-$. The product is zero if the resulting
diagram has an arc which is not attached to the lines $x=0$ or
$x=1$, called \emph{floating} arc, Figure  ~\ref{dang2}. Also,
if $y \in {}_m B^-_n$, $z\in {}_k B^-_l$ and $n\not= k$, then the
concatenation is not defined and we set $yz=0.$ Thus, for any two
elements $y,z$ of $B^-$ the product $yz$ is either $0$ or an
element of $B^-$.

\begin{figure} \scalebox{0.9}{\includegraphics{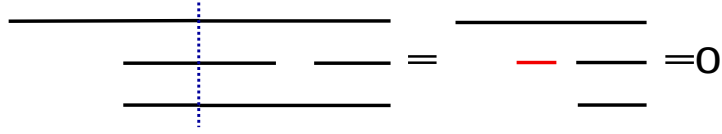}}
\caption{Concatenation of these two diagrams equals zero since the
resulting diagram contains a floating arc.}
\label{dang2}
\end{figure}

\begin{remark}
Alternatively, we can avoid drawing sarcs, and instead draw just
their endpoints on the vertical lines $x=0,1.$ Then the product of
two diagrams is zero if the composition has an isolated point in
the middle of the diagram.
\end{remark}
The composition induces an associative $\fieldk$-algebra structure on $A^-$.
For each $n$ there exists a unique diagram in $_nB^-_n$ without
sarcs. We denote this diagram and its image in $A^-$ by $1_n.$
These elements are minimal idempotents in $A^-.$

We have
\begin{equation*}A^-=\displaystyle{\bigoplus_{n,m\ge 0}\hspace{0.02in} _n A^-_m},
\end{equation*}
where $_n A^-_m$ is the vector space with the basis $_n B^-_m$.
$A^-$ is a non-unital associative algebra with a system of mutually orthogonal
idempotents $\{ 1_n\}_{n\ge 0}.$ We consider left modules $M$ over
$A^-$ with the property
$$ M =\displaystyle{ \oplusop{n\ge 0} 1_n M}. $$
This property is analogous to the unitality condition $1M = M$ for modules over
a unital algebra. For a module $M$, we write  $M^m$ for the direct sum of $m$
copies of $M$.

Let $P_n=A^- 1_n$ be the projective $A^-$-module $P_n$  with a basis consisting of all
diagrams in $B_n^-$. Define $M_n$, called the \emph{standard module}, as the quotient
of $P_n$ by the submodule spanned by all diagrams which have right sarcs.
Therefore, a basis of $M_n$ is the set of diagrams in $B^-_n$ with no right sarcs. In particular,
if $1_mM_n\not= 0$ then $m\ge n.$ Notice that $b \cdot a =0$ for any $a \in M_n$
and a diagram $b \in B^-$ with at least one right sarc, Figure  ~\ref{SLARCSonMn}.

\begin{figure}[h]
\begin{center}
\scalebox{.85}{
\includegraphics{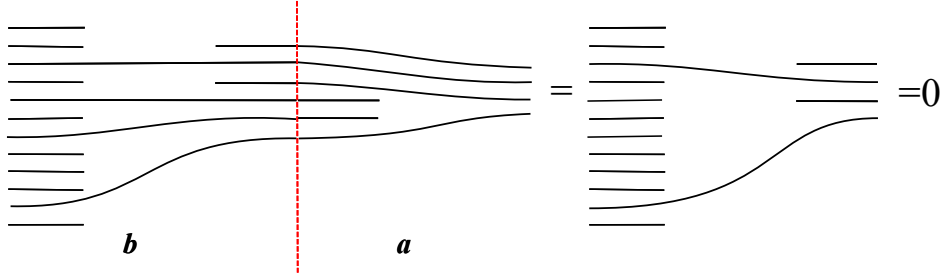}}
\end{center}
\caption{For any diagram $a$ representing an element of a standard
module and a diagram $b \in B^-$ with right sarcs the product $b
\cdot a=0$.}
 \label{SLARCSonMn}
\end{figure}

A left $A^-$--module $M$ is called finitely--generated if for some finite subset
$\{m_1,m_2, \dots, m_k \}$  of $M$ we have $M = A^-m_1 + \dots +A^-m_k$.
$M$ is finitely generated if and only if it is a quotient of
 $\displaystyle{\bigoplus_{n=0}^N} P_n^{a_n}$ for some $a_n
 \geq 0$, $N \in \mathbb{N}.$

Let $A^-\dmod$ be the category of finitely-generated left $A^-$-modules and $A^-\pmod$ the
category of finitely-generated projective left $A^-$-modules.

\begin{prop}
The hom space $\Hom_{A^-}(M',M'')$ is a finite-dimensional
$\fieldk$-vector space for any $M',M'' \in A^-\dmod.$
\end{prop}

\begin{proof} It is sufficient to consider the case $M'=P_n.$ We have
$\Hom(P_n,M'') = 1_n M''.$ But $1_nM''$ is finite-dimensional, since $M''$
is a quotient of finite direct sum of $P_m$'s and, $1_n P_m$ is
finite-dimensional.
\end{proof}

\begin{cor} The category $A^-\dmod$ is Krull-Schmidt.\end{cor}

Let $L_n=\fieldk 1_n$ be the one-dimensional module over $A^-$ on which any
element of $B^-$ other than $1_n$ acts by zero.

\begin{lemma}\label{SLARCSimpleLn}
 Any simple $A^-$-module is isomorphic to $L_n$ for some $n \leq 0$.
\end{lemma}
\begin{proof}
Let $L$ be a simple $A^-$-module and $I$ the $2$-sided ideal in $A^-$ spanned by all diagrams
with at least one left sarc. Notice that $1_nI^{n+1}=0$ for all $n \geq 0$. Since  $IL$ is a submodule
of $L$, then either $IL=L$ or $IL=0$. If $IL=L$ then $I^mL=L$ for every $m$ and $0=1_nI^{n+1}L=1_nL$ for all $n$, a contradiction.
Hence $IL=0$ and every simple module $L$ is actually an $A^-/I$-module. The algebra $A^-/I$ is
directed, in the sense that
\begin{eqnarray*}
  1_n(A^-/I) 1_m &=& 0 \rm{\,\,\,\, if \,} n>m,\\
 1_n(A^-/I) 1_n &\cong & \fieldk .
\end{eqnarray*}
Hence, $\oplusop{k \leq n} 1_kL$ is a submodule of $L$ for every $n$. With $L$ being simple,  $1_nL=L$ for some $n$, and $L$ is one-dimensional, isomorphic to $L_n$. \end{proof}

\begin{theorem} \label{SLARCProjModDec}
Any finitely--generated projective left  $A^-$-module $P$ is isomorphic to a finite
direct sum of indecomposable projective modules $P_n$,
$$ P \cong  \oplusoop {n=0}{N}P_n^{a_n}.$$
The multiplicities $a_n \in \Z_+$ are invariants of $P$.
\end{theorem}
\begin{proof}
 The module $P_n$ is indecomposable, since its
endomorphism ring $R=\Hom_{A^-}(P_n,P_n)$ is local. Indeed, the
diagrams in ${}_n B_n$ other than $1_n$ span a 2-sided ideal
$J$ in $R$ and $J^N=0$ for $N$ sufficiently large.
Therefore $J$ is the radical of $R$,
$R/J\cong \fieldk$, 
 and $R$ is local.

Take a finitely--generated projective $A^-$-module $P$ and any maximal proper submodule $Q$. The
simple module $P/Q$ is isomorphic to $L_n$, for some $n$. Surjections
$$ P \stackrel{p_1}\lra L_n \stackrel{p_2}\longleftarrow P_n $$
lift to homomorphisms $P \stackrel{\alpha}\ra P_n \stackrel{\beta}\ra P.$
\[
\xy (0,0)*{P}="A"; (20,20)*{P_n}="B"; (20,0)*{L_n}="C";
{\ar@{->>}^{p_2} (20,20)*{P_n}; (20,0)*{L_n}};
 {\ar@{->>}^{p_1} (0,0)*{P};(20,0)*{L_n}};
 {\ar^{\beta}(20,20)*{P_n};(0,0)*{P}};
 {\ar@/^1pc/^{\alpha} (0,0)*{P};(20,20)*{P_n}};
\endxy
\]
Notice that $p_1 \beta \alpha=p_1$ and $p_2
\alpha \beta =p_2$ which gives $p_2(\alpha \beta -1)=0$. Hence $1-
\alpha \beta \in J(End(P_n)),$ the Jacobson radical of the endomorphism ring,
and there exist an integer $N$ such
that $(1-\alpha \beta)^N=0.$ Thus, there exist an endomorphism $\delta$  of $P_n$
 such that
$1-\alpha \beta \delta=0.$ Hence for $\beta'=\beta \delta$ we get $\alpha \beta'=1$ which means
$$P \cong Im \beta \oplus Ker \alpha \cong P_n \oplus Ker
\alpha'$$ i.e. that $P_n$ is direct summand of $P$. Proceeding by induction, we get
 $P \cong \oplusoop {n=0}{N}P_n^{a_n}$.
The Krull--Schmidt property implies that multiplicities $a_n$ are invariants of $P.$
\end{proof}

The projective module $P_n$ has a filtration by standard modules
$M_m$, over $m\le n.$ Specifically, consider the filtration
\begin{equation}
   P_n = P_n(\leq n)\supset P_n( \leq n-1) \supset \dots \supset P_n( \leq 0)=0,
   \label{SLArcsFilterPn}
\end{equation}

\noindent where $P_n( \leq m)$ is spanned by the diagrams in $B^-_n$ of
width at most $m$ (equivalently, with at least $n-m$ right sarcs).
Left multiplication by a basis vector cannot  increase the width,
hence $P_n( \leq m)$ is a submodule of $P_n$.
 The quotient $P_n( \leq m)/P_n(\leq m-1)$ has a basis of diagrams of
 width exactly $m.$ These
diagrams can be partitioned into $\sbinom{n}{m}$ classes
enumerated by positions of the $n-m$ right sarcs. The
quotient $P_n( \leq m)/P_n(\leq m-1)$ is isomorphic to the
direct sum of $\sbinom{n}{m}$ copies of the standard module $M_m.$
Consequently, we have an equality in the Grothendieck group
of the additive category $A^- \dmod$:
\begin{equation}
[P_n]= \sum_{m=0}^n \sbinom{n}{m} [M_m]. \label{SARCSPnMminG}
\end{equation}

Next, we prove that the non-unital algebra $A^-$ is Noetherian, hence the category $A^- \dmod$ is abelian.

\begin{prop} \label{SLARCSNoetherian}
A submodule of a finitely-generated left $A^-$-module is
finitely-generated.
\end{prop}

\begin{proof} Any finitely generated $A^-$-module  is a quotient of
$\displaystyle{\oplusoop{i=0}{N} P_i^{n_i}}$ for some $N$ and some $n_0, n_1,
\ldots, n_N$, hence it suffices to show $\displaystyle{\oplusoop{i=0}{N} P_i^{n_i}}$
is N\"{o}etherian. Furthermore it is enough to show that any
submodule of $P_n$ is finitely-generated. Since $P_n$ has a finite
filtration by standard modules, it suffices to check that any submodule of a
standard module $M_n$ is finitely-generated.
The induction base, case $n=0$ is trivial, since $M_0=\displaystyle{\bigoplus_{m \geq
 0}1_mM_0}$, each term $1_mM_0$ is one-dimensional and generates a
 submodule of finite codimension in $M_0.$

 Basis elements $b$ of $M_n$ can be labeled by length $n+1$ sequences
of non-negative integers $(a_1,a_2, \ldots,a_{n+1})$.  Here $a_1$
is the number of sarcs below the bottom larc and $a_{n+1}$ is the
number of sarcs above the top larc. Each $a_i$, $2 \leq i \leq n$
represents the number of sarcs between ($i-1$)-st and $i$-th larc,
counting larcs from bottom to top (Figure ~\ref{SLARCSBasisMn}).

\begin{figure}[h]
\begin{center}
\scalebox{.8}{\includegraphics{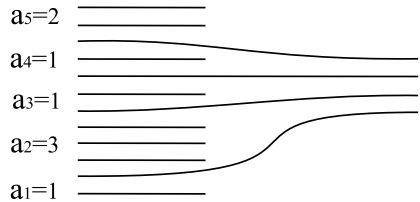}} \caption{Basis
element for $M_n$.} \label{SLARCSBasisMn}
\end{center}
\end{figure}

We call $a_{n+1}$ the degree $deg(b)$ of a basis element
$b=(a_1,a_2, \ldots,a_{n+1})\in M_n.$ Degree of an arbitrary
element $d=\displaystyle{\sum_{i} x_ib_i }\in M_n$, $x_i \in \fieldk^*$ is equal to
$deg\,d=\displaystyle{\max_{i} \,deg(b_i)}.$ For an element
$d=\displaystyle{\sum_{i} x_ib_i }\in M_n$ define an element
$d'=\displaystyle{\sum_{deg\,b_i=deg\,d}} x_ib_i \in M_n$, which
is a sum of terms of $d$ with the highest degree.

\begin{figure}[h]
\begin{center}
\scalebox{.9}{\includegraphics{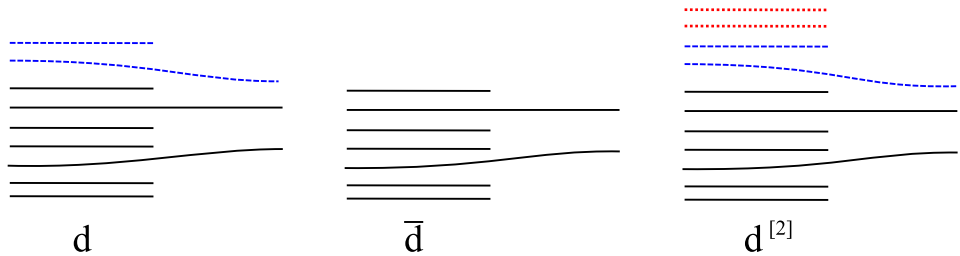}}
\caption{This figure shows element $d\in M_3$, the corresponding
$\overline{d}$ and the element obtained by degree shift 2 denoted
by $d^{[2]}$. The top larc and sarcs above it are denoted by
dashed lines. Two added sarcs in $d^{[2]}$ are shown as dotted
lines.} \label{SLARCSCoeff1}
\end{center}
\end{figure}

Let $\overline{d} \in M_{n-1}$ be the element obtained from an
element $d \in M_n$ by removing the top larc and all of the sarcs
above it in each of the diagrams in $d$. Moreover, we define an
element $d^{[p]}\in M_n$ obtained
from  $d$ by adding $p$ sarcs on the top of each
diagram in $d$. In particular, $deg\,d^{[p]}=deg\,d+p$  (Figure
~\ref{SLARCSCoeff1}).
\begin{figure}
\begin{center}
\scalebox{.9}{\includegraphics{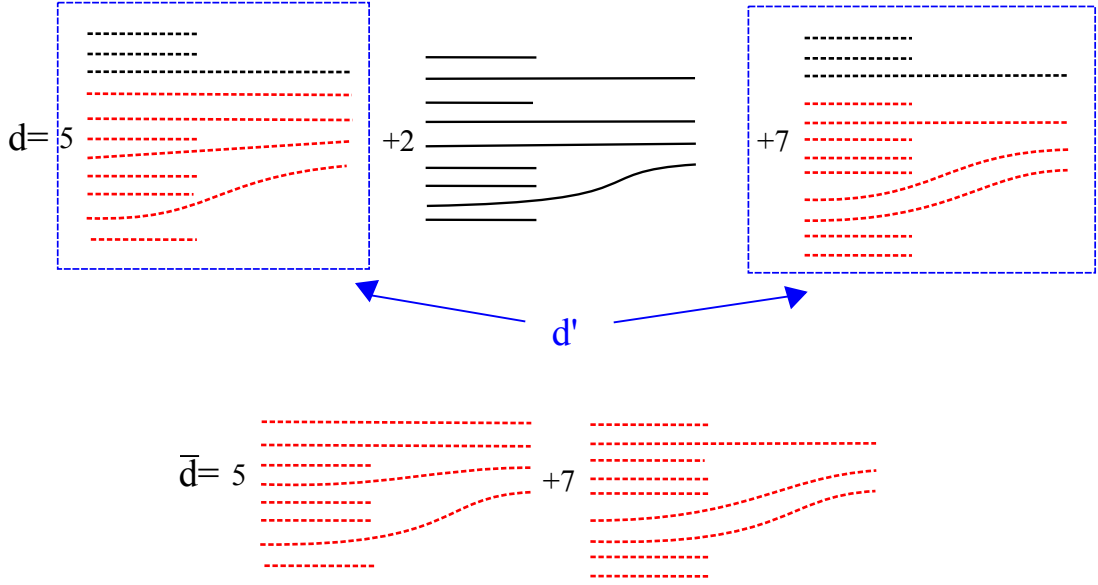}}
\caption{Highest degree summands of the element $d\in M_4$ are
contained in the top left and right rectangles.
The bottom picture shows $\overline{d}$.} \label{SLARCSCoeff}
\end{center}
\end{figure}
To continue with the proof, let $M$ be any submodule of $M_n$ and
$d_0$ be an element of the least degree in $M$. Assuming that
$d_0, \ldots, d_k$ have already been chosen, take $d_{k+1}\in M
\setminus (d_0, \ldots, d_k)$ where $(d_0, \ldots, d_k)$ is the submodule
generated by the least degree elements. Continuing by
induction we obtain a sequence of elements $d_i\in M$.

Let $c_i:=\overline{d'_i} \in M_{n-1}$ and $\overline{M}$ denote
the submodule of $M_{n-1}$ generated by $c_i$'s. According to the
induction hypothesis $M_{n-1}$ is N\"{o}etherian, hence
$\overline{M}=(c_0,c_1,\ldots)$ must be finitely generated. In
other words, there exists $N \in \mathbb{N}$ such that
$\overline{M}=(c_0,c_1,\ldots, c_N)$.

Assume that $M \neq (d_0, \ldots, d_N)$. Then there exist
$d_{N+1}\in M \backslash(d_0, \ldots, d_N)$, and
$c_{N+1}=\displaystyle{\sum_{k=0}^N \alpha_k c_k}$ for some
$\alpha_k \in A^-$.
 Let $d^*= \displaystyle{\sum_{k=1}^N \alpha_k d_k^{[deg d_{N+1}-degd_k]}}$.
Now $d_{N+1}-d^* \notin\,(d_0,\ldots, d_N)$ and $deg
(d_{N+1}-d^*)<deg(d_{N+1})$ which contradicts the minimality of
$deg(d_{N+1}).$ Therefore $M =(d_0, \ldots, d_N)$ and $M_n$ is
Noetherian.\footnote{This proof is analogous to the proof that
$\fieldk[x_1,x_2, \ldots, x_n]$ is Noetherian.}
\end{proof}

The involution of the set $B^-$ which reflects a diagram about a
vertical axis takes ${}_n B^-_m$ to ${}_m B^-_n$ and induces an
anti-involution of $A^-.$ Hence the ring $A^-$ is right
N\"{o}etherian as well.

\begin{definition}
  Grothendieck group $K_0(A)$ of finitely generated projective
  $A$-modules is an abelian group generated by symbols $[P]$ of
 finitely-generated projective left $A$ modules
 $P$, with defining relations  $[P]=[P']+[P'']$ if $P \cong P' \oplus P''.$
\label{SLARCGroth}
\end{definition}
\begin{prop}\label{PropK0}
  $K_0(A^-)$ is a free abelian group with basis $\{ [P_n] \}_{n \geq 0}$.
\end{prop}
Proposition \ref{PropK0} follows from Theorem \ref{SLARCProjModDec}.

Observe that the existence of the filtration (\ref{SLArcsFilterPn}) of projective
modules $P_n$ by standard modules $M_m$ implies that $M_m$ has a finite projective
resolution $P(M_m)$ by $P_n$'s, for $n\leq m.$ Consequently, we can view $M_m$ as
an object of the category $\mc{C}(A^- \pmod)$ of bounded complexes of
finitely-generated projective $A^-$-modules. Morphisms in this category are
homomorphisms of complexes modulo zero-homotopic homomorphisms.
Grothendieck groups of categories $A^-\pmod$ and $\mc{C}(A^- \pmod)$
 are canonically isomorphic:
$$ K_0(\mc{C}(A^-\pmod))\cong K_0(A^-\pmod)$$ via the isomorphism taking the symbol of
\begin{center}
$Q=(\dots\ra P^i \ra P^{i+1}\ra \dots) \in \mc{C}(A^- \pmod)\,$ to $\,[Q]=\displaystyle{\sum_{i\in \Z} (-1)^i [P^i]}\in K_0(A^-).$
\end{center}
 Hence, the equality ~(\ref{SARCSPnMminG}) can be interpreted within $K_0(A^-).$

The transformation matrix from the basis of the symbols $[P_n]$ of indecomposable
projective modules to the basis of symbols $[M_m]$ of standard modules is
upper-triangular, with ones on the diagonal and nonzero
coefficients being the binomials $\sbinom{n}{m}.$ The entries of the inverse
matrix are $(-1)^{n+m}\sbinom{n}{m}.$ Thus we
have the following equation in $K_0(A^-)$:
\begin{equation}
[M_n] = \sum_{m=0}^{n} (-1)^{n+m}\sbinom{n}{m} [P_m].
\label{SLARCSMnPminG}
\end{equation}

We identify the projective Grothendieck group $K_0(A^-)$ with
$\Z[x]$ by sending the symbols of projective modules $[P_n]$ to monomials $x^n$,
 and define an inner product on the basis $\{x^n\}_{n \geq 0}$ by
\begin{equation}
(x^n, x^m)=\dim\,\Hom(P_n,P_m)=|{}_nB^-_m|=\sbinom{n+m}{m}
\end{equation}

This identification will be justified in Section \ref{Tensor} by introducing a monoidal
structure on $A^- \pmod$ under which $P_n \otimes P_m \cong P_{n+m}.$

Under this identification, equation (\ref{SLARCSMnPminG}) gives
\begin{equation}
[M_n]=\displaystyle{\sum_{m \leq n}(-1)^{n+m}\sbinom{n}{m}x^m}=(x-1)^n,
\end{equation}
so the symbols of standard modules  $[M_n]$ correspond to $(x-1)^n.$

Equation (\ref{SLARCSMnPminG}) hints at the existence of a projective resolution
of $M_n$ which starts with $P_n$ and has $\sbinom{n}{m}$ copies of
$P_m$ in the $(n-m)$-th position:
\begin{equation}
0\ra P_0 \ra \dots \ra P_{n-m}^{\sbinom{n}{m}} \ra \dots \ra
P_{n-2}^{\sbinom{n}{2}}\ra P_{n-1}^{\sbinom{n}{1}} \ra P_n
\ra M_n \ra 0 \label{SLARCSResMnbyPn}
\end{equation}
Let us construct this resolution.

\begin{figure}[h]
\scalebox{1}{\includegraphics{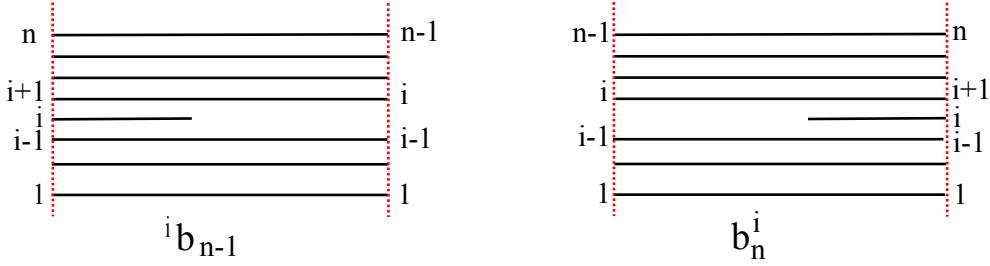}}
\caption{Diagrams  ${}^ib_{n-1} $ and $b_n^i$ used in defining differentials in projective resolution of standard modules and resolution of simple by standard modules.}
\label{SLARCSDifferentialResMn}
\end{figure}

  Denote the diagram with $n$ larcs and one left sarc at the
$i$-th position by  ${}^ib_{n-1} \in  {}_{n}B^-_{n-1}$.   The diagram
obtained from ${}^ib_n$ by a reflection along the vertical axis
is denoted by $b_n^i \in {}_{n-1}B^-_{n}$, Figure  \ref{SLARCSDifferentialResMn}.
The product of ${}^{i}b_{n-1}$ or $b_n^{i}$ with an arbitrary diagram
$a \in B^-$, when defined and non-zero, differs from the diagram $a$
in the following way (see Figure \ref{SLARCSBasicD}):

\begin{enumerate}
   \item $a\cdot{}^{i_j}b_n$ turns $i_j$th larc in a diagram $a$ into left sarc,
   \item ${}^{i_j}b_n \cdot a$ adds left sarc between $i$th and
      $i+1$-st larc in $a$,
   \item $a \cdot b_n^{i_j}$ adds right sarc between $i$th and
     $i+1$-st larc in $a$,
   \item $b_n^{i_j} \cdot a$ turns $i_j$th larc  in a diagram $a$ into right sarc.
\end{enumerate}

\begin{figure}[h]
\scalebox{1}{\includegraphics{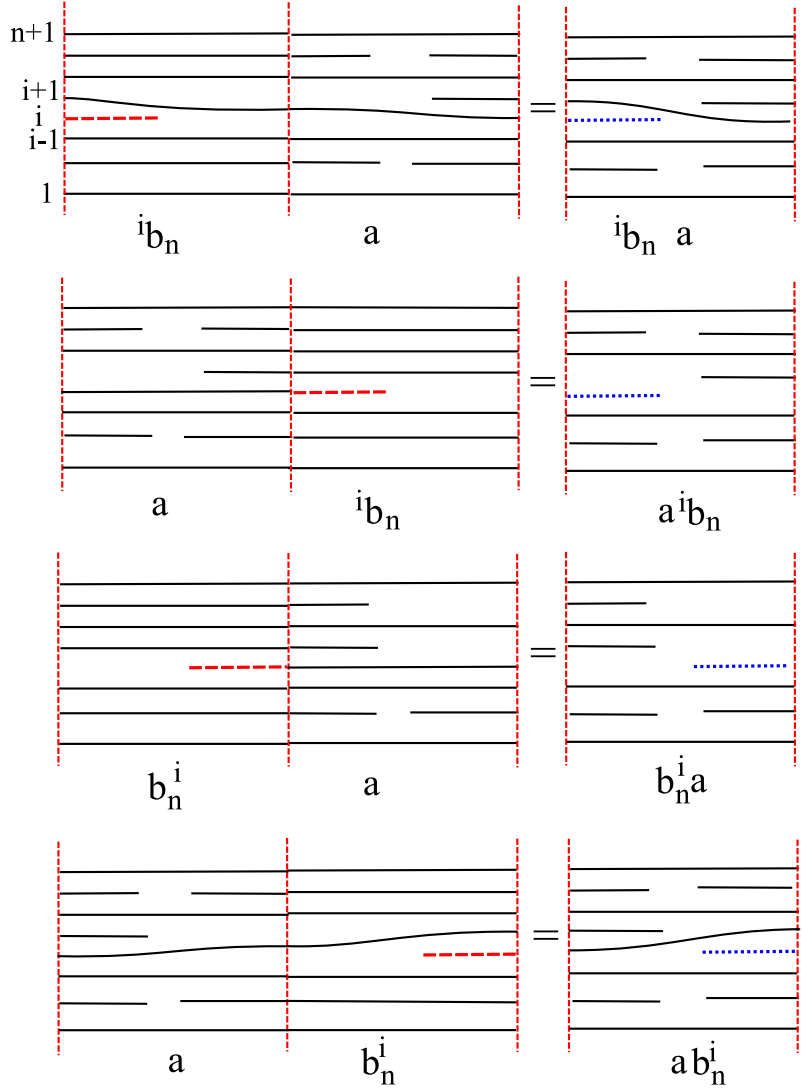}} \caption{Diagrams
${}^{i}b_n$ and $b_n^{i}$ and their products with a
diagram $a \in B^-$. Dashed line represents the difference between
them and the diagram $1_n$ and the dotted line in the resulting
diagram emphasizes the difference between diagram $a$ we started with and
the product diagram.} \label{SLARCSBasicD}
\end{figure}

Let $I_{m}=\{i_1,\ldots,i_m \} \subseteq \{1, \ldots, n \}$, $i_1<\dots<i_m $ be a subset
of cardinality $m \leq n$. Label the summands of the $m$-th term  $P_{n-m}^{\sbinom{n}{m}}$
by these subsets $I_m$,  $P_{n-m}^{I_m}$. Let $I_{m,l}:= I_m \setminus \{i_l\}$.
Removing an element $i_l$ of a set $I_m$ can be interpreted as composing a diagram in $B_{n-m}$
on the right with a diagram $b^{p}_{n-m+1}$, obtained in the following way. Take a diagram
$b_n^{i_l}$ and delete all long arcs at positions labeled by elements in $I_{m,l}$,
resulting in a diagram  $b^{p}_{n-m+1}$, where $p$ denotes a position of $i_l$ in the ordered set
 $\{1,2, \ldots, n\} \setminus I_m \cup \{i_l\}$,  Figure ~\ref{SLARCSDifferentialResMn}.

\begin{figure}[h]
\scalebox{1.35}{\includegraphics{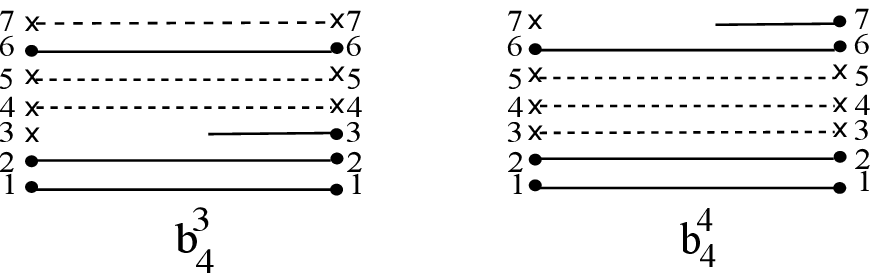}}
\caption{Differentials $d_{\{3,4,5,7\},+1}$ and $d_{\{3,4,5,7\},+4}$ in the projective resolution of
standard module $M_7$.}
\label{SlarcsDifferentialNov}
\end{figure}

Next, define the differential $$d:P_{n-m}^{ \sbinom{n}{m}}\lra P_{n-(m-1)}^{\sbinom{n}{m-1}}$$
 as the sum $$d=\displaystyle{\sum_{I_m} \sum_{l=1}^m d_{I_{m,+l}}}.$$
of maps $d_{I_m,+l}:P_{n-m}^{I_m}\ra P_{n-(m-1)}^{I_{m, l}}$
sending $a \in P_{n-m}^{I_m}$ into $d_{I_{m,+l}}(a)=(-1)^{l-1}a \cdot b^{p}_{n-m+1}$,
 For example, Figure  \ref{SlarcsDifferentialNov} shows how to define the differentials
 $d_{\{1,3,4,5\},+5}$ and $d_{\{1,3,4,5\},+1}$ in the resolution of
$M_7$ sending $P_3^{\{1,3,4,5\}}$ into $P_4^{\{1,3,4\}}$ and $P_4^{\{3,4,5\}}$, respectively.

\begin{prop}\label{SLARCSThmPnMn}
The complex  ~(\ref{SLARCSResMnbyPn}) with the differential defined above is exact.
\end{prop}
\begin{proof}
  \begin{figure}[h]
\scalebox{.95}{\includegraphics{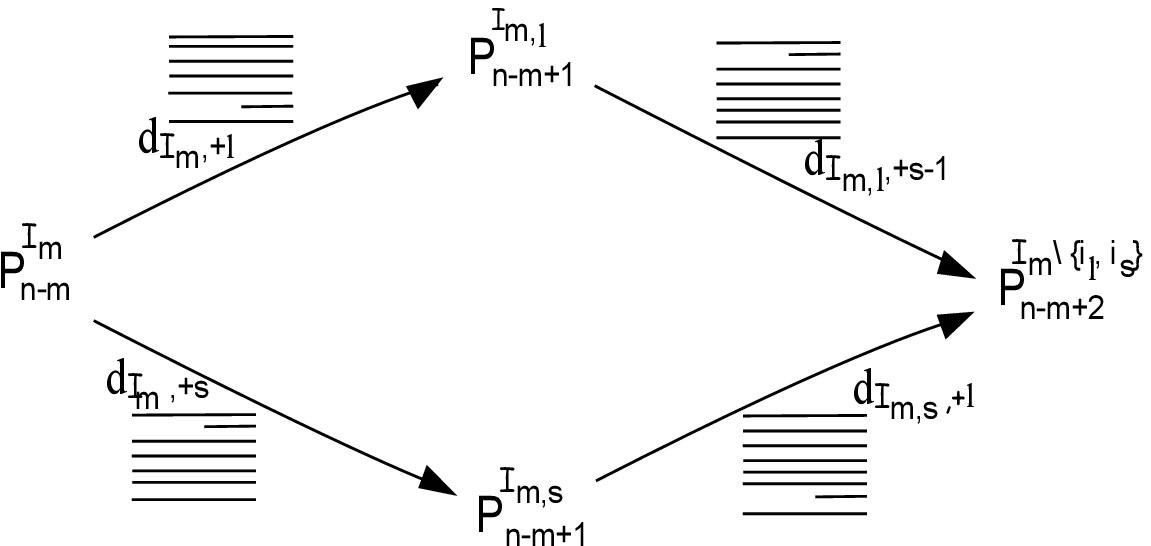}}
 \caption{Commutative diagram for the projective resolution of standard modules.}
\label{SLARCSProjResStand}
 \end{figure}
The proof that $d^2=0$  follows from the sign convention and the commutative diagram on Figure  ~\ref{SLARCSProjResStand} which shows $d_{I_{m,l},+s-1} \cdot d_{I_m,+l} =d_{I_{m,s},+l} \cdot d_{I_m,+s}$, for $l<s.$
 The proof that ~(\ref{SLARCSResMnbyPn}) is exact uses a slight
generalization of this square. Viewed as a complex of vector spaces,  ~(\ref{SLARCSResMnbyPn}) splits into the sum of
complexes:
\begin{equation*}
0 \ra 1_p P_0 \ra \dots \ra 1_p P_{n-m}^{\sbinom{n}{m}} \ra \dots \ra
1_p P_{n-2}^{\sbinom{n}{2}}\ra 1_p P_{n-1}^{\sbinom{n}{1}} \ra 1_pP_n
\ra 1_p M_n \ra 0
\end{equation*}
 one for each element of $_{k}B^-_n$, $k \leq n$ with no left
sarcs. Each of the complexes in the sum is isomorphic to the total complex of
an $p$--dimensional cube with a copy of ground field $\mathbf{k}$
in each vertex and each edge an isomorphism. Hence, all complexes are contractible.
\end{proof}

A finite--dimensional $A^-$--module $M$ has a finite filtration with simple modules $L_n$
as subquotients.
Due to one-dimensionality of $L_n$ the multiplicity of $L_n$ in $M$, denoted by $[M:L_n]$,
 equals dim$1_nM$. A finitely-generated $A^-$--module $M$ is not necessarily finite
 dimensional but it satisfies the following property
 \begin{equation*}
   dim(1_nM)< \infty, \, \rm{ for } \, n \ge 0,
 \end{equation*}
which we call a locally finite--dimensional property.

 For locally finite--dimensional  module $M$ we define the multiplicity of $L_n$ in $M$ as:
 \begin{equation*}
   [M:L_n] := dim(1_nM).
   \label{SLARCSMultiplicity}
\end{equation*}

This definition is compatible with the usual notion of multiplicity  of
$L_n$ in $M$ as the number of times $L_n$ appears in the composition series
of $M$ when $M$ is finite--dimensional.

 Let us now specialize to standard modules $M_m$. We have
\begin{equation}
  [M_m:L_n]=  dim(1_nM_m)=
 \left\{
  \begin{array}{ll}
    \sbinom{n}{m}, & \hbox{ for $n\geq m$;} \\
    0, & \hbox{if $n<m$.}
  \end{array}
\right.
\end{equation}

Recall that $[P_n:M_m]= \sbinom{n}{m} $, hence
\begin{equation}
 [P_n:M_m]= [M_m:L_n].
 \label{SLARCSBGG}
\end{equation}
Thus, our diagrammatically defined algebra possesses the
Bernstein--Gelfand--Gelfand (BGG) reciprocity property. Indecomposable projective
modules  $P_n$ have filtration by standard modules $M_m$, with $m
\leq n$ and $[P_n:M_n]= 1.$  The
multiplicity in the RHS in the equality (\ref{SLARCSBGG})is understood in the generalized sense,
as explained above.

Define the Cartan matrix $C(A^-)$ by
\begin{equation}
C(A^-)_{i,j}:= dim {\Hom} (P_i, P_j)
  \label{SLARCSCartan}
\end{equation}
and by $m(A^-)$ the multiplicity matrix $m(A^-)_{i,j}:=[P_i: M_j]=[M_j:L_i].$
 Then we have the following equality:
\begin{equation}
  C(A^-)= m(A^-)m(A^-)^t.
\end{equation}
Indeed,
\begin{eqnarray*}
  C(A^-)_{i,j}&=& dim {\Hom} (P_i, P_j)=[P_i : L_j] \\
              &=& \sum_k [P_i: M_k][M_k : L_j] = \sum_k m(A^-)_{i,k} m(A^-)_{j,k}  \\
              &=& \sum_k m(A^-)_{i,k} m(A^-)^t_{k,j}= (m(A^-)m(A^-)^t)_{i,j}.
\end{eqnarray*}

\begin{prop}
$\mathrm{Ext}^i(M_n, M_m)=(1_{n-i}M_m)^{\sbinom{n}{i}}$.
\end{prop}
\begin{proof}
Since the map between $\Hom(P_k,M_m)$ and $\Hom(P_{k-1},M_m)$ induced by the differential
in the projective resolution of standard module $M_n$ is trivial, proof follows from the
fact that $\Hom(P_k,M_m)=\Hom(A^-1_k, M_m)=1_kM_m.$
\end{proof}

\begin{prop}\label{Prop27}
\begin{equation}
\mathrm{Ext}^i(M_n, L_m)\cong
\left\{%
\begin{array}{ll}
       \mathbf{k}^{\sbinom{n}{n-m}} & \hbox{ {\rm if } $m \leq n, i=n-m$;} \\
   0 & \hbox{{\rm otherwise.}} \\
\end{array}%
\right.
\end{equation}
\end{prop}
\begin{proof}
Obviously,  $\mathrm{Ext}^i(M_n, L_m)=0$ for $m>n$. To compute
$\mathrm{Ext}^i(M_n, L_m)$ we use projective resolution
~(\ref{SLARCSResMnbyPn}) and get the complex:
\begin{equation}
0 \leftarrow \Hom(P_0,L_m) \leftarrow \ldots
\leftarrow  \Hom(P_{n-1},L_m)^{\oplus n} \leftarrow \Hom(P_n,L_m)
\leftarrow 0 \label{SLARCSMnLm}
\end{equation}
Notice that $\Hom(P_{n-k},L_m)=\left\{%
\begin{array}{ll}
    \mathbf{k}, & \hbox{if $m=n-k$;} \\
    0, & \hbox{otherwise.} \\
\end{array}%
\right.    $\\
 In the case $m=n-k$, $k \in \mathbb{Z}_+$, complex
~(\ref{SLARCSMnLm}) will be nontrivial  only in
degree $n-m$, and
$H^{n-m}=\mathbf{k}^{\sbinom{n}{n-m}}=\mathrm{Ext}^{n-m}(M_n,
L_m)$. All other Ext's are zero.
\end{proof}

\begin{prop}
Homological dimension of slarc algebra standard module $M_n$ is
$n$.
\end{prop}
\begin{proof}
Projective dimension of $M_n$ is at most $n$ as we have
constructed a projective resolution ~(\ref{SLARCSResMnbyPn}) of
that length. For $m=0$, Proposition \ref{Prop27} says
that $Ext^n(M_n, L_0)= \mathbf{k}$, hence the projective dimension is equal to $n$.
\end{proof}

Next we construct a resolution of a simple module $L_k$ by standard modules
$M_m$ for $m\ge k:$
 \begin{equation}
 \stackrel{d}\lra  M_{k+m}^{ \sbinom{k+m}{m}}\stackrel{d}\lra \dots \stackrel{d}\lra  M_{k+2}^{ \sbinom{k+2}{2}} \stackrel{d}\lra M_{k+1}^{\sbinom{ k+1}{1}}
\stackrel{d}\lra M_k \stackrel{d}\lra L_k \lra 0.
\label{SLARCSResLnByMn}
 \end{equation}
 Let $I_m=\{i_1,i_2, \ldots, i_m\}$ be a subset of $ \{1,2, \ldots,n \}$,
 $m \leq n$, $i_1<i_2< \ldots< i_m$.  Let
 $I_{m,-p}$ denote the set obtained from $I_m$ by removing the $p$--th element and subtracting $1$ from
 all subsequent elements:
 \begin{equation}\label{defIm}
   I_{m,-p}= \{i_1,i_2, \ldots,i_{p-1}, i_{p+1}-1, \ldots, i_{m}-1\} = I_m \setminus \{i_p \}
 \end{equation}

 The $m$-th term of the resolution is a direct sum $M_{k+m}^{\sbinom{k+m}{m}}$
of standard modules $M_{k+m}$. On the level of diagrams, multiplicity $\sbinom{k+m}{m}$ represents the
number of ways to add $m$ right sarcs to a diagram in $M_k$ to obtain a diagram in $M_{k+m}$.
Let $I_m=\{i_1,i_2, \ldots, i_m\} \subseteq \{1,2, \ldots,k+m \}$ be the set
describing positions of added larcs. Each summand  $M_{k+m}^{I_m}$ is
labeled by one of these subsets, and the differential will take summand labeled
by $I_m$ into summands labeled by $I_{m,-l}$, for $0<l \leq m$, by composing on the right with diagrams containing a single short right arc and no left sarcs, see Figure  ~\ref{SLARCSDifferentialResMn}.

 More precisely, let us define maps $$d_{I_m,-l}: M_{k+m}^{I_m} \xrightarrow{^{l} b_{k+m-1} } M_{k+m-1}^{I_{m,-l}}$$
 that send  $a \in M_{k+m}^{I_m}$ into $$d_{I_m,-l}(a)= (-1)^l\, a\, \cdot {}^{l}b_{k+m-1}$$
 where  the diagram ${}^{l}b_k$ is shown on the Figure \ref{SLARCSDifferentialResMn}.
 The differential $$d:M_{k+m}^{\sbinom{k+m}{m}} \rightarrow M_{k+m-1}^{\sbinom{k+m-1}{m-1}}$$ is an alternating sum
of these maps $$d=\displaystyle{\sum_{I_m} \sum_{l=1}^m (-1)^l d_{I_m,-l}}.$$

\begin{figure}[h]
\scalebox{1.2}{\includegraphics{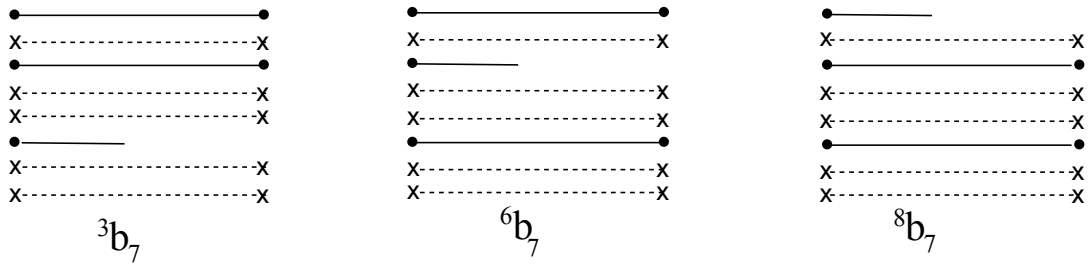}}
\caption{Examples of diagrams used in defining differential maps $d_{\{3,6,8\},-1}$,
$d_{\{3,6,8\},-2}$ and $d_{\{3,6,8\},-3}$ in the resolution of simple module $L_5$ by standard modules.}
\label{SlarcsDiffNovMnLm}
\end{figure}

 For example, diagrams on Figure ~\ref{SlarcsDiffNovMnLm} show how to define the differentials'  maps $d_{\{3,6,8\},-1}$, $d_{\{3,6,8\},-2}$ and $d_{\{3,6,8\},-3}$ in the resolution of
$L_5$ sending $M_8^{\{3,6,8\}}$ into $M_7^{\{5,7\}}$, $M_7^{\{3,7\}}$, and $M_7^{\{3,6\}}$.
In general, for a map $d_{I_m, -l}$, $0<l\leq m$, sending $M_{n+1}^{I_{m}} \ra M_n$ in the resolution of $L_{n+1-m}$, start with a diagram $1_{n+1}$, turn arc $i_l$ into a short right arc, then remove all long arcs labeled by numbers which are not in $I_m={i_1, i_1, \ldots, i_m}$, shown in dotted lines on Figure ~\ref{SlarcsDiffNovMnLm}.

\begin{prop}\label{proofLnByMn}
The complex ~(\ref{SLARCSResLnByMn}) with the differential defined above is exact.
\end{prop}

\begin{proof}
The proof that $d^2=0$ is the same as in Proposition \ref{SLARCSThmPnMn}, except that the differential is defined
using diagrams that lower the number of larcs, see Figure \ref{SLARCSDifferentialResMn} and Figure \ref{SLARCSProjResStand}.

\begin{figure}[h]
\begin{center}
\scalebox{0.6}{\includegraphics{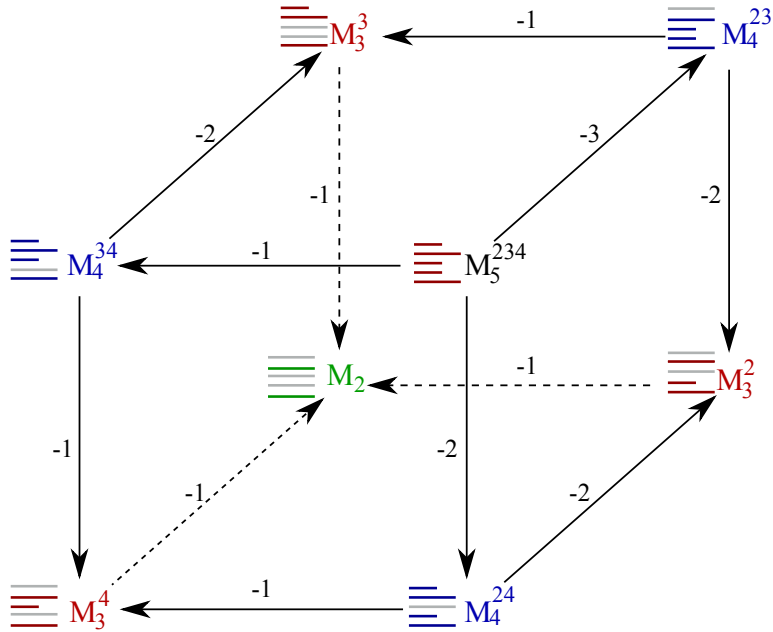}}
\end{center}
\caption{ A $3$-dimensional cube
 in the resolution of $L_2$, corresponding to $M_5^{\{2,3,4\}}$, where label ${\{2,3,4\}}$ describes a diagram in $B_2$ with
$3$ left short arcs and the remaining two larcs shown to the left of the symbol $M_5$.}
\label{SLARCSExactLkMn}
 \end{figure}

To prove the exactness, notice that complex (\ref{SLARCSResLnByMn}) splits into the sum of complexes of vector spaces
\begin{equation*}
  1_nM_n^{{n \choose n-k}} \ra 1_nM_{n-1}^{{n-1 \choose (n-1)-k}}\ra \dots \ra 1_nM_{k+1}^{{k+1 \choose 1}}\ra 1_nM_k
\end{equation*}
for each $n>0.$ In turn, each of these complexes splits into the sum of $(n-k)$-dimensional cubes,
corresponding to diagrams in ${}_nB_{n-k}$ with $k$ larcs, $n-k$ left sarcs and no
right sarcs, containing a copy of the field $\fieldk$ in each vertex.
For example, the resolution of $L_2$ contains a summand of corresponding to $M_5^{\{2,3,4\}}$ represented by a total complex of a $3$-dimensional cube shown on Figure \ref{SLARCSExactLkMn}. Sets labeling the vertices denote positions of short arcs in the corresponding diagrams
shown on the left side of the module symbol. Arrows are labeled with positions of elements which are being removed.

\end{proof}
 Informally, on the level of Grothendieck groups we have the
 following relation:
\begin{eqnarray*}
[L_n]&=& \sum_{k=0}^{\infty} (-1)^{k}{n+k \choose k} [M_{n+k}]\\
     &=& \sum_{k=0}^{\infty} (-1)^{k}{n+k \choose k}(x-1)^{n+k}=
     \frac{(x-1)^n}{x^{n+1}}.
\label{SLARCLnMminformal}
\end{eqnarray*}
We will not try to make sense out of this infinite sum.

In order to obtain projective resolution of a simple module $L_n$ we
construct a bicomplex, see Figure \ref{SLARCSBi},  with a projective resolution (\ref{SLARCSResMnbyPn}) of $M_{n+k}$, $k \geq 0$ lying above each copy of a standard module in
the resolution (\ref{SLARCSResLnByMn}) of $L_n$ by
standard modules $M_m$, $m\ge n$.

\begin{figure}[h]

\tiny{
\[\begin{CD}
@. @VV{d_M}V \\
\ldots @>d_H>>  P_{m-2}^{ {n+m \choose m}{n+m \choose n-2}} \\
@. @VV{d_M}V  @VV{d_M}V\\
 \ldots @>d_H>> P_{m-1}^{ {n+m \choose m}{n+m \choose n-1}}  @>d_H>> \ldots  @>d_H>>   P_0^{ n+1}\\
@. @VV{d_M}V @VV{d_M}V  @VV{d_M}V \\
@>d_H>>  \ldots @>d_H>>  \ldots  @>d_H>> P_1^{(n+1)^2} @>d_H>>  P_0  \\
@. @VV{d_M}V @VV{d_M}V  @VV{d_M}V @VV{d_M}V \\
\ldots @>d_H>>  \ldots @>d_H>> \ldots  @>d_H>> \ldots @>d_H>>  \ldots \\
@. @VV{d_M}V @VV{d_M}V @VV{d_M}V  @VV{d_M}V  \\
\ldots @>d_H>>  P_{n+m-1}^{ {n+m \choose m} {n+m \choose 1}}
@>d_H>>  \ldots   @>d_H>>P_n^{(n+1)^2}  @>d_H>>  P_{n-1}^{ n}  \\
@. @VV{d_M}V @VV{d_M}V   @VV{d_M}V @VV{d_M}V \\
\ldots @>d_H>> P_{n+m}^{ {n+m \choose m}} @>d_H>>\ldots  @>d_H>>  P_{n+1}^{n+1}  @>d_H>>  P_n \\
@. @VV{d_M}V @VV{d_M}V    @VV{d_M}V @VV{d_M}V\\
\ldots @>d_L>> M_{n+m}^{ {n+m \choose m}} @>d_L>> \ldots @>d_L>> M_{n+1}^{n+1}  @>d_L>> M_n @>d_L>> L_n  @>d_L>> 0 \\
 @. @VVV    @VVV   @VVV  @VV{d_M}V \\
0 @>>>   0 @>>>  0 @>>>  0 @>>>  0
\end{CD}\]\\}
\caption{Bicomplex, whose total complex is a projective resolution of $L_n$.}
 \label{SLARCSBi}
\end{figure}

To complete the construction of the bicomplex, we define the horizontal differential denoted by
$d_H$. Each copy of the projective module $P_{n+m-k}$ in the bicomplex shown in Figure \ref{SLARCSBi}
comes with a pair of labels $P_{n+m-k}^{I_{m+n},J_k}$. The first label $I_{n+m}$ is equal to the
label of the standard module $M_{n+m}$ in the resolution of $L_n$, and $J_k$ is the label of $P_{n+m-k}$
in the projective resolution of $M_{n+m}$.

  Horizontal differential
  $d_H:P_{n+m-k}^{ {n+m \choose m} {n+m \choose k}} \lra P_{n+(m-1)+k}^{ {n+m-1 \choose m-1} {n+m-1 \choose k}}$
  is a signed sum of maps $d_{I_{m+n},J_k}$ sending $ a \in P_{n+m-k}^{I_{m+n},J_k}$ to

\begin{equation}
 d_{I_{m+n},J_k}(a)=\displaystyle{\sum_{\begin{smallmatrix}
p=0\\
 i_p \notin J_k
\end{smallmatrix}
}^{n+m}
(-1)^{i_p-1}\, a \, {}^{i_p}b} \in \displaystyle{\bigoplus_{
\begin{smallmatrix}
p=0\\
 i_p \notin J_k
\end{smallmatrix}}^{n+m}}P_{n+m-1-k}^{I_{m+n,-p},J_{k,-p}}
 \label{SLARCSbicomplexDiff}
\end{equation}

where $I_{m+n, -p}$ and $J_{k,-p}$ are defined in (\ref{defIm}).

\begin{figure}[h] \small{
\[\begin{CD}
P_{n+m-(k+1)}^{\oplus \sbinom{n+m}{m} \sbinom{n+m}{k+1}}
@>d_H>>
P_{n+(m-1)-(k+1)}^{\oplus \sbinom{n+m-1}{k+1} \sbinom{n+m-1}{k+1}} \\
@VV{d_M}V @VV{d_M}V \\
 P_{n+m-k}^{\oplus \sbinom{n+m}{m} \sbinom{n+m}{k}}
 @>d_H>>
P_{n+(m-1)-k}^{\oplus \sbinom{n+m-1}{m-1} \sbinom{n+m-1}{k}}  \\
 \end{CD}\]\\}
\caption{An anticommutative square in the bicomplex on Figure \ref{SLARCSBi}.} \label{SLARCSZoomIn}
\end{figure}

\begin{prop}\label{SLARCSbicomplex}
The diagram on Figure \ref{SLARCSBi} is a bicomplex -- all squares are
anticommutative.
\end{prop}
\begin{proof}
Direct computation, see Figure \ref{SLARCSZoomIn}.
\end{proof}

 The projective resolution
\begin{equation}
P(L_n): \dots \ra C_{n,t} \ra C_{n,t-1} \ra \dots \ra C_{n,0} \ra L_n \ra 0
\end{equation}
 of the simple module $L_n$ is defined in the following way:
\begin{equation}
 C_{n,t}=\displaystyle{\bigoplus_{\begin{smallmatrix}
 m+k=t \\
 n+m \geq k
\end{smallmatrix}  }} P_{n+m-k}^{{n+m \choose m}{n+m \choose k}}
\label{SLARCSLnByPn}
\end{equation}

The total differential $d_t$ is a sum of the horizontal differential
$d_H$
, and the vertical differential
$d_M$ in the projective resolution of standard modules:
\begin{equation*}
  d_t=d_H+d_M.
\end{equation*}

In other words, the resolution (\ref{SLARCSLnByPn}) is the total
complex of the bicomplex in Figure (\ref{SLARCSBi}). Since
each column in the bicomplex is exact, the following proposition holds:

\begin{prop}\label{SLARCSResLnExact}
The chain complex ~(\ref{SLARCSLnByPn}) is exact.
\end{prop}

\begin{prop}\label{SLARCSDimLn}
Simple modules $L_n$ over slarc algebra $A^-$ have infinite
homological dimension.
\end{prop}
\begin{proof}
Based on the resolution by projective  modules
~(\ref{SLARCSLnByPn}), it is sufficient to show that $Ext^i(L_n,M)$
is nontrivial for arbitrarily large $i \in \mathbb{N}$ and some
$A^- \dmod$ module $M$. Recall that
\begin{equation*}
  \Hom(P_i,L_m)=\left\{%
\begin{array}{ll}
    \mathbf{k}, & \hbox{$m=i$;} \\
    0, & \hbox{otherwise.} \\
\end{array}%
\right.
\end{equation*}

$C_{n,t}$ contains all $P_i$  for ${\rm max}(0,n-t)\leq i <n+t$ such that $n+t-i \equiv 0$ (mod $2$).
Let $M=L_0$ and notice that $P_0 \in C_{n,t}$ for every $t \geq n$ such that $n+t$ is even.
 Hence, the chain complex built out of homomorphism spaces $\Hom(C_{n,t},L_0)$ with the
 differential induced from the resolution reduces to the infinite cochain complex
 having trivial groups in odd degrees and non-trivial groups in even degrees for $t \geq n:$

$$Ext^{n+t}(L_n,L_0) \cong \Hom(C_{n,t},L_0) \cong\left\{%
\begin{array}{ll}
    \mathbf{k}, & \hbox{$t=n$;} \\
    \mathbf{k}^{\sbinom{\tfrac{t+n}{2}}{\tfrac{t-n}{2}}}, & \hbox{$t+n$ \, even, \, $t>n$.} \\
\end{array}%
\right. $$
Therefore, $Ext^{n+t}(L_n,L_0)$ is non-trivial for arbitrarily
large $t>n$ such that $n+t$ is even.
\end{proof}
Slarc algebra $A^-$ can be viewed as a graded algebra with the
grading defined by the total number of sarcs in a diagram. In
particular, if we regard ~(\ref{SLARCSLnByPn}) as the
graded resolution, the differential is increasing the degree by
$1$.

\begin{cor}\label{SLARCSKoszul}
The algebra of slarcs $A^-$ is Koszul.
\end{cor}
\section{Functors}
\vspace{0.5cm}
\begin{center}\textbf{Approximations of the identity}
\end{center}
\vspace{0.5cm}

Recall that  $B^-(\leq k)={\displaystyle \bigsqcup_{i=0}^k B^-(i)}$ denotes
diagrams in $B^-$ of width less than or equal to $k$. Let $A^-
(\leq k)$, $k \geq 0$ denote the subspace of $A^-$ spanned by
diagrams in $B^- ( \leq k)$. This subspace is an $A^-$--subbimodule of $A^-$.
 Let $A^-(k)$ be the quotient subbimodule
$A^-(\leq k)/ A^-(\leq k-1)$. Let ${}_nP$ denote a right
projective module ${}_nP=1_nA^-$ and, analogously to the standard
modules $M_n$, let ${}_nM$ be the quotient of ${}_nP$ by the submodule spanned by
all diagrams with a left sarcs. One can think of diagrams of
${}_nM$ as reflections along vertical axis of diagrams in $M_n.$
\begin{figure}[h]
\begin{center}
\scalebox{.8}{
\includegraphics{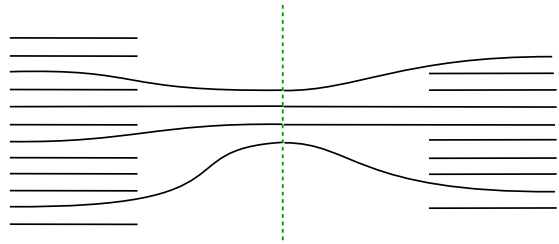}}
\caption{Diagram in $B^-(4)$ viewed as a product of elements in
$M_{4}$ and ${}_{4}M$.}
 \label{SLARCSBasisNk}\end{center}
\end{figure}

\begin{prop}
  $A^-(\leq k)/A^-(\leq k-1) \cong M_k \otimes_{\mathbf{k}} {}_kM$ as $A^-$-bimodules (Figure  \ref{SLARCSBasisNk}).
\end{prop}

For a given $k \geq 0$ , define a right exact functor $F_k: A^-\dmod \ra A^-\dmod$ by
$$F_k(M)=A^-(\leq k) \otimes_{A^-} M,$$
for an $A^-$-module $M.$
The image of the standard module $M_m$ under functor $F_k$ is:
\begin{equation}
A^-(\leq k)\otimes_{A^-} M_m=\left\{%
\begin{array}{ll}
    M_m, & \hbox{if $k \geq m$;} \\
    0, & \hbox{otherwise.} \\
\end{array}%
\right.
\end{equation}
By definition $P_m=A^-1_m$, hence $A^-(\leq k)\otimes_{A^-}P_m=A^-(\leq
k)\otimes_{A^-}A^-1_m=A^-(\leq k)1_m$, and this is a submodule of
$P_m$ spanned by diagrams of width less than or equal to $k$:

\begin{equation}\label{SLARCSIdPn}
F_k(P_m)=A^-(\leq k)\otimes_{A^-}P_m=\left\{%
\begin{array}{ll}
    P_m, & \hbox{if $k \geq m$;} \\
    P_m(\leq k), & \hbox{if $k<m$.} \\
\end{array}%
\right.
\end{equation}

Recall that in the Grothendieck group, projective modules $P_n$
correspond to $x^n$ and standard modules $M_n$ to $(x-1)^n$. Modules $P_n(\leq k)$
 have finite homological dimension, since they admit finite filtrations with
 successive quotients isomorphic to standard modules.
Therefore, functor $F_k$ descends to an operator on the Grothendieck group $K_0(A^-)$, denoted by $[F_k]$.
The action of $[F_k]$  on $[P_n]=\displaystyle{\sum_{m=0}^n \sbinom{n}{m}[M_m]}$ is equal to:

\begin{equation}
[F_k][P_n]=\left\{%
\begin{array}{ll}
    [P_n]=x^n, & \hbox{if $k \geq n$;} \\
   \displaystyle{ \sum_{m=0}^k \sbinom{n}{m}[M_m]=\sum_{m=0}^k \sbinom{n}{m}(x-1)^m}, & \hbox{if $k<n$.} \\
\end{array}%
\right.
\end{equation}

In other words, for  $k \geq n$ operator $[F_k]$ acts via
identity on $[P_n]$, and for $k<n$ it approximates identity and can be
viewed as taking the first $k+1$ terms $\displaystyle{\sum_{m=0}^k \sbinom{n}{m}[M_m]}$
in the expansion of $[P_n]$ in the basis $\{(x-1)^m \}_{m \geq 0}.$

\begin{proposition}
  Higher derived functors of the functor $F_k$ applied to a standard module are zero:
  \begin{equation*}
    L^i F_k(M_n)=\left\{
                   \begin{array}{ll}
                     M_n, & \hbox{$i=0, k\geq n$;} \\
                     $0$, & \hbox{{\rm otherwise}.}
                   \end{array}
                 \right.
     \end{equation*}
\end{proposition}
\begin{proof}
Projective resolution $P(M_n)$ has the form (\ref{SLARCSResMnbyPn}):
\begin{equation}\label{last}
0\ra P_0 \ra \dots \ra P_{n-m}^{\sbinom{n}{m}} \ra \dots \ra
P_{n-2}^{\sbinom{n}{2}}\ra P_{n-1}^{\sbinom{n}{1}} \ra P_n \ra 0
\end{equation}
Terms in this resolution are multiples of projective modules $P_m$, for $m
\leq n.$ Based on (\ref{SLARCSIdPn}), if $k \geq n$, $F_k$ acts as identity
on the resolution, implying the proposition in this case. Assume now that $k<n$.
The differential in (\ref{SLARCSResMnbyPn}) applied to a diagram in any $P_{n-m}$
preserves the width of the diagram, and (\ref{SLARCSResMnbyPn}) splits, as a complex
of vector spaces, into a direct sum of complexes over all widths from $0$ to $n$.
These complexes are exact unless the width is exactly $n$; in the latter case the
summand is isomorphic to $0 \ra M_n \ra 0.$

Applying $F_k$ to the resolution (\ref{last}) produces the complex
\begin{equation}
0\ra P_0 \ra \dots \ra P_{n-m}^{\sbinom{n}{m}}(\leq k) \ra \dots \ra
P_{n-2}^{\sbinom{n}{2}}(\leq k) \ra P_{n-1}^{\sbinom{n}{1}}(\leq k) \ra P_n (\leq k)
\ra 0
\end{equation}
which is exact for  $k \leq n$, being a direct sum of exact complexes over all widths from $0$ to $k.$
\end{proof}


\vspace{1cm}
 \begin{center}\textbf{Restriction and induction functors and what they categorify}
\end{center}
\vspace{0.5cm}

For a unital inclusion $\iota: B \hookrightarrow A$ of arbitrary rings the
induction functor
$$Ind: B-mod \ra A-mod$$
given by $Ind(M)=A \otimes_B M$ is left adjoint to the restriction functor,
$$ \Hom_A(Ind(M),N) \cong \Hom_B(M,Res(N)). $$
If the inclusion is non-unital, i.e., $\iota$ takes the unit
element of $B$ to an idempotent $e\not= 1$ of $A$, the restriction
functor needs to be redefined: to an $A$-module $N$ assign an
$eAe$-module $eN$ and then restrict the action to $B$. The
induction functor is defined as before, but now
$$ Ind(M)=A\otimes_B M
\cong (Ae \otimes_B M) \oplus (A(1-e)\otimes_B M) =
Ae \otimes_B M, $$ and the induction is still left adjoint to the
restriction. A similar construction works for non-unital $B$ and
$A$ equipped with systems of idempotents.

We now specialize to slarc algebra  $A^-$ and the inclusion
$\iota: A^- \hookrightarrow A^-$ induced by adding a straight
through line at the top of every diagram, i.e. diagram $d \in {} _m
B_n$ goes to $ \iota(d) \in{}_{m+1} B^-_{n+1}$. In particular, the
system of idempotents $\{1_n\}_{n \geq 0}$ goes to $\{1_{n+1}\}_{n
\geq 0}$ missing $1_0$. This inclusion $\iota$ gives rise to both
induction and restriction functors, with
\begin{eqnarray}
  Ind(N) &\cong& A^- \otimes_{\iota(A^-)} N \\
  Res(N) &\cong& N/1_0 N\cong \oplusop{k>0}1_k\,N {\rm \,\,with \,algebra\, } A^- {\rm \, acting \, on \, the \, left \, via \,}
\iota.
\end{eqnarray}

  In particular,  $1_{n-1}Res(M) \cong 1_nM$.

Notice that for simple modules
\begin{equation*}
  Res (L_n)=\left\{
                     \begin{array}{ll}
                       L_{n-1}, & \hbox{if $n> 0$;} \\
                       0, & \hbox{$n=0$,}
                     \end{array}
                   \right.
\end{equation*}

while $Ind(L_{n})$ is an infinite-dimensional module
such that
\begin{equation*}
  1_m (Ind(L_{n}))=\left\{
                     \begin{array}{ll}
                       \fieldk, & \hbox{if $m > n$;} \\
                       0, & \hbox{otherwise.}
                     \end{array}
                   \right.
\end{equation*}

\begin{figure}[h]
\begin{center}
\scalebox{.8}{
\includegraphics{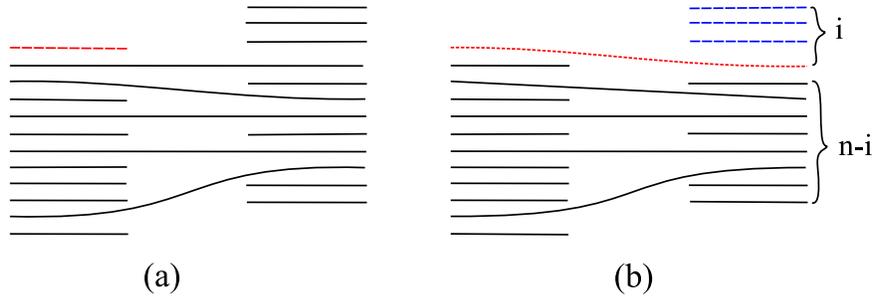}}
\caption{Decomposition of the module $P_n$ as a sum of vector
spaces spanned by diagrams of type (a) where left sarc is attached
to the top left point and type (b) where the top left point is
connected by larc to the i-th point on the right. In particular,
diagram in (a) is an element of $P_{12}^{\emptyset}$ and (b) belongs to $P_{12}^{(i)}$. }
 \label{SLARCSResPnBase}\end{center}
\end{figure}

\begin{prop}$Res(M_n) \cong M_{n} \oplus M_{n-1}$
for $n > 0$, and $Res (M_0) \cong M_0$.
\end{prop}
\begin{proof}
Let $M_n^L$ and $M_n^{\emptyset}$ denote spans
 of diagrams in $M_n$ with the top left point being a part of a left sarc
or a larc, respectively (diagrams in Figure  \ref{SLARCSResPnBase} can be treated as elements
of standard modules if we delete right returns). Then $Res(M_n) \cong M_n^L \oplus
M_n^{\emptyset}$ as left $A^-$-modules. Furthermore, $M_n^{\emptyset}\cong M_n$
and  $M_n^L \cong M_{n-1}.$
\end{proof}

\begin{prop}
$Res(P_n) \cong \displaystyle{\oplusoop{k=0}n}P_k$ for all $n \geq 0.$
\end{prop}

\begin{figure}[h]
\begin{center}
\scalebox{.8}{
\includegraphics{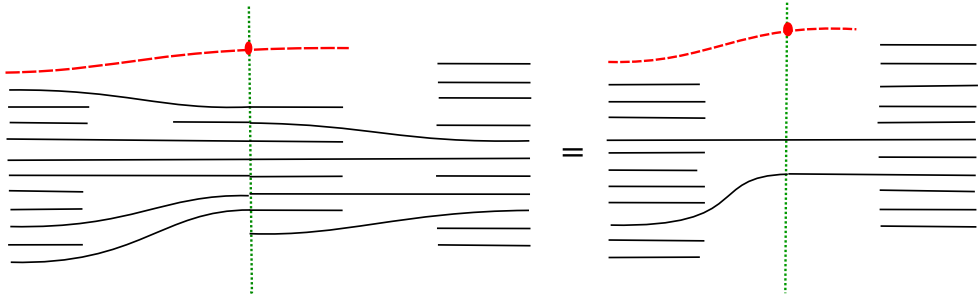}}
\caption{$P^{\emptyset}$ is isomorphic to projective module $P_n.$ }
 \label{SLARCSResPnAction1}
 \end{center} \end{figure}

\begin{proof}
For each $i \geq 1$, let $P_n^{(i)}$ denote spans of diagrams in
$P_n$ with top left point connected by a larc to the $i$-th point
on the right and by $P_n^{\emptyset}$ the span of diagrams such that at the top
we have a left sarc (Figure \ref{SLARCSResPnBase}). Each of these spans is a direct
summand of $Res(P_n)$.

 Then $Res (P_n) \cong P_n^{\emptyset}
\oplus \displaystyle{\bigoplus_{i=1}^{n}} P_n^{(i)}$ as left
$A^-$-modules. It is easy to see that $P_n^{\emptyset}\cong
P_n$ (Figure \ref{SLARCSResPnAction1}) since the top left sarc is
fixed. Similarly, $P_n^{(i)} \cong P_{n-i}$
since $i-1$ top right sarcs are fixed (Figure
\ref{SLARCSResPnAction}).
\end{proof}

\begin{figure}[h]
\begin{center}
\scalebox{.8}{
\includegraphics{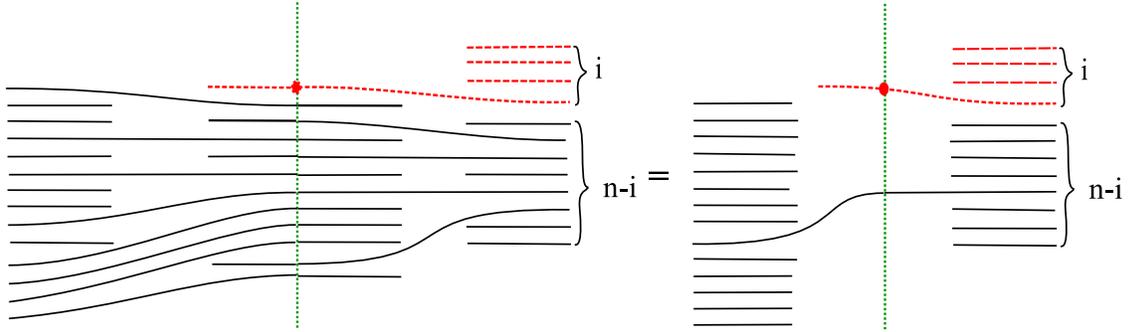}}\caption{$P_n^{(i)}$ is isomorphic to projective module $P_{n-i}.$ }
 \label{SLARCSResPnAction}
 \end{center}
 \end{figure}

\begin{prop}$Ind(P_n) \cong P_{n+1}$
for $n \geq 0$.
\end{prop}
\begin{proof}
Follows from the definition of the induction functor, also see Figure
\ref{SLARCSIndPn}.\end{proof}

\begin{figure}[h]
\begin{center}
\scalebox{.8}{
\includegraphics{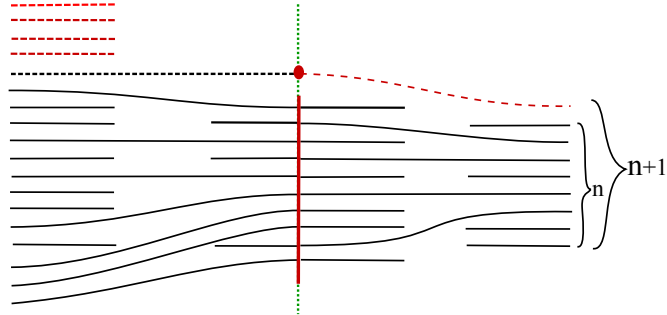}}
\caption{Induction on projective modules: an element of the tensor product
$A^- \otimes_{\iota (A^-)}P_n$ is presented diagrammatically
by composing basis elements of $A^-$ and $P_n$, which can
exchange elements of $\iota (A^-)$ through the red vertical line.}
 \label{SLARCSIndPn}
 \end{center}
 \end{figure}

\begin{prop}
For $n \geq 0$ there exist a short exact sequence:
\begin{equation}
0 \ra M_n \ra  Ind(M_n) \ra  M_{n+1} \ra 0.
\end{equation}
\end{prop}

\begin{proof}
Notice that the right action of $\iota(A^-)$ fixes the top right point of a
diagram in $A^-$. 
Depending on whether this point has a right sarc or larc attached to it, see Figure \ref{SLARCSIndMn},
we get a copy of $M_n$ or $M_{n+1}$ as a submodule or a quotient of $Ind(M_n)$, respectively.\end{proof}
\begin{figure}[h]
\begin{center}
\scalebox{.7}{
\includegraphics{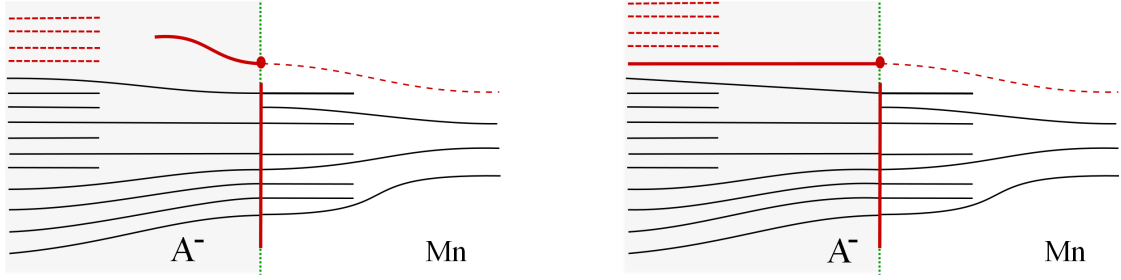}}
\caption{Induction on standard modules.}
 \label{SLARCSIndMn}
 \end{center}
 \end{figure}

\begin{proposition}
  Higher derived functors of the induction functor applied to a standard module are zero:
  \begin{equation*}
    L^i Ind(M_n)=0, \, {\rm for\, every} \, i > 0.
  \end{equation*}
\end{proposition}
\begin{proof}
  Induction functor applied to the projective resolution (\ref{SLARCSResMnbyPn}) of the standard module $M_n$ gives:
  \begin{equation*}
    0 \ra P_1 \ra P_2^{\sbinom{n}{1}} \ra \ldots \ra P_m^{\sbinom{n}{m-1}}\ra \ldots \ra P_n^{\sbinom{n}{n-1}} \ra P_{n+1} \ra 0
  \end{equation*}
  where the differential corresponds to the one from projective resolution (\ref{SLARCSResMnbyPn}) with a long arc added on top of each diagram.
  This complex splits, as a complex of vector spaces, into the sum of two copies of the original complex depending on whether the top arc is a larc or right sarc.
\end{proof}

On the Grothendieck group induction corresponds to the multiplication by $x$ as:
 \begin{eqnarray*} \,[P_n]=x^n &\mapsto& [P_{n+1}]=x^{n+1}
\\  \,[M_n]=(x-1)^n & \mapsto &
[M_n]+[M_{n+1}]= (x-1)^n +(x-1)^{n+1}.
\end{eqnarray*}

 On the other hand, restriction (always exact) takes:
\begin{eqnarray*}
 \,[P_n]=x^n &\mapsto& \displaystyle{\sum_{i=0}^n} [P_i]= \displaystyle{\sum_{i=0}^n }x^i \\
 \,[M_n]=(x-1)^n & \mapsto & \displaystyle{\sum_{i=0}^n}[M_i]+[M_{i-1}]= \displaystyle{\sum_{i=0}^n}(x-1)^n +(x-1)^{n-1}.
\end{eqnarray*}

On the Grothendieck group $[Res]$ acts by sending
  \begin{eqnarray*}
 \,f(x) \mapsto \frac{x f(x)-f(1)}{x-1}.\\
\end{eqnarray*}

\vspace{1cm}\begin{center}\textbf{Tensor products }\label{Tensor}
\end{center}
\vspace{0.5cm}
We define the tensor product bifunctor $$ A^- \pmod \times A^- \pmod \ra A^- \pmod$$  on
indecomposable projective modules by $P_n \otimes P_m = P_{n+m}$ and extend it to all objects using Theorem ~\ref{SLARCProjModDec}.
Next, define tensor functor on basic morphisms of projective modules $\alpha: P_n \ra P_{n'}$ and $\beta: P_m \ra P_{m'}$, where
$\alpha \in {}_nB_{n'}$, $\beta \in {}_mB_{m'}$ by placing $\alpha$ on top of $\beta$ (see Figure  \ref{SLARCSTensorMor}) and then
extending it to all morphisms and objects using bilinearity.

\begin{figure}[h]
\begin{center}
\scalebox{1.1}{
\includegraphics{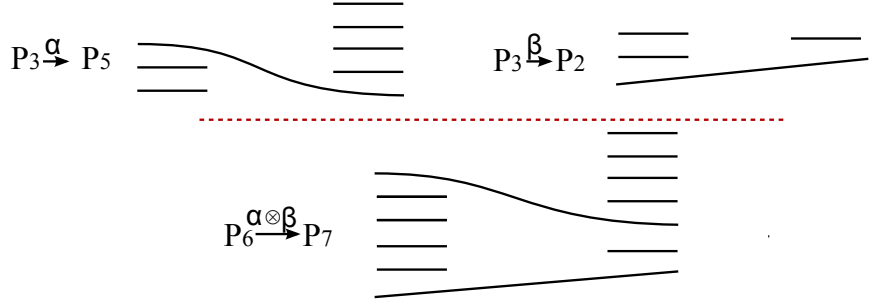}}
\caption{Tensor product defined on basic morphisms of projective modules.}
 \label{SLARCSTensorMor}\end{center}
\end{figure}

Tensor product extends to a bifunctor $C(A^- \pmod) \times C(A^- \pmod) \ra C(A^- \pmod).$
Hence, $A \pmod$ and $C(A^- \pmod)$ are monoidal categories.
Since standard modules have finite projective
resolutions, they can be viewed as objects of $C(A^- \pmod)$. Let $P(M_n)$ be the projective resolution (\ref{SLARCSResMnbyPn})
 of a standard module $M_n$.

 Note that in the Grothendieck group $[M_n]=(x-1)^n$ and $$[M_n]\cdot[M_m]=(x-1)^{n+m}=[M_{n+m}].$$

 One can
 guess now that this equality may lift to the category $A^- \dmod$ or $C(A^- \pmod)$, and we show that it does.

\begin{lemma}In $C(A^- \pmod)$, $P(M_n) \otimes P(M_m) \cong P(M_{m+n})$ for $m,n \geq 0.$
\end{lemma}
\begin{proof}
  The $p$-th term in the product of the projective resolutions $P(M_m)$ and $P(M_n)$ is
  $$ \displaystyle{\bigoplus_{k+l=p}}P_k^{\sbinom{n}{k}}\otimes P_l^{\sbinom{m}{l}} \cong P_p^{\sbinom{n+m}{p}}$$
  This module isomorphism respects differentials and gives an isomorphism of complexes.
\end{proof}

\begin{cor}\label{tensorMn}
The following relation holds between standard modules viewed as objects of $C(A^- \pmod)$
  $$M_n \otimes M_m \cong M_{m+n}.$$
\end{cor}
On Grothendieck group the tensor product descends to the multiplication in the ring $\Z[x]$, under the
  isomorphism of abelian groups $K_0(A^-) \cong \Z[x].$

To define tensor product for arbitrary modules we need to construct and tensor their projective
resolutions. If both modules $M,N$ have finite filtrations with successive quotients isomorphic to
standard modules $M_n$ for various $n$, then the derived tensor product $M \widehat{\otimes} N$  has
cohomology only in degree zero, and $H^0(M \widehat{\otimes} N)\cong_{D^b} M \widehat{\otimes} N$
has a filtration by standard modules. Derived tensor product restricts to a bifunctor on the
category of modules admitting a finite filtration by standard modules.

\vspace{1cm}\begin{center}
  \textbf{Cabling functors}
\end{center}
\vspace{0.5cm}

For every $A^-$--module $M$ and a positive integer $k$ construct
the corresponding cabled module ${}^{[k]}M$ in the following way:
\begin{equation}\label{SLdefCable}
  1_n{}^{[k]}M=1_{nk}M, \, {\rm hence }\, {}^{[k]}M= \displaystyle{\oplusop{n \geq 0}1_{nk}M}.
\end{equation}

\begin{figure}[h]
\begin{center}
\scalebox{.7}{
\includegraphics{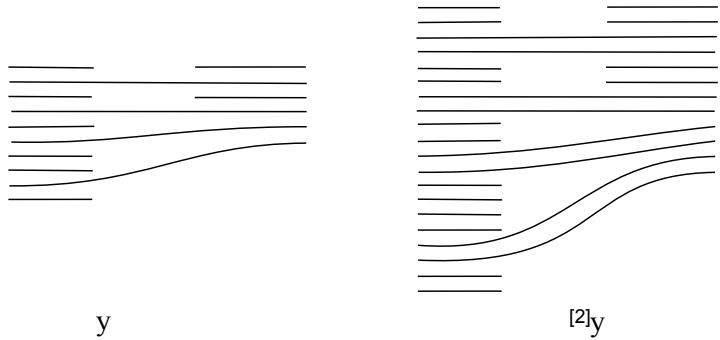}}
\caption{A diagram $y \in {}_{11}B^-_{6}$ and 2-cable ${}^{[2]}y \in
{}_{22}B^-_{12}$.} \label{SLARCSCableY}
  \end{center}
\end{figure}

 Given a diagram $y \in {}_sB^-_{l}$,
construct a diagram ${}^{[k]}y \in {}_{sk}B^-_{lk}$, called the $k$-cabling of $y$, by taking $k$ parallel
copies of each arc (Figure ~\ref{SLARCSCableY}). For example, ${}^{[k]}1_n=1_{nk}.$
 By definition, the action of an element $ \alpha \in A^- $ on
${}^{[k]}M_n$ is the regular action of its $k$-cabling $\alpha^k$.

What is the result of $k$-cabling simple, standard
and projective modules? It is easy to see that, if $k$ divides $n$,
the $k$-cabling of the simple module $L_n$ is the module
$L_{n/k}:$
\begin{equation}
1_m{}^{[k]}L_n=1_{km}L_n=\left\{%
\begin{array}{ll}
    \fieldk, & \hbox{if $km=n$;} \\
    0, & \hbox{otherwise.} \\
\end{array}%
\right.
 \label{SLARCSAntiBraidLn}
\end{equation}
If $k$ does not divide $n$ the result is zero, ${}^{[k]}L_n=0.$

Recall that basis elements of standard $A^-$ modules $M_n$ correspond to diagrams in $B_n$
with $n$ through arcs and an arbitrary number of left sarcs.
Let $S(n,k,i)$ denote the number of ways to select $n$ numbers between $1$ and
$ki$ such that each of the sets $\{kj+1, \ldots, k(j+1) \}_{0 \leq j
<i}$ contains at least one of the selected numbers.
\begin{proposition}\label{SLARCSCablingM}
  \begin{equation*}
  {}^{[k]}M_n \cong \displaystyle{\oplusoop{i=\lceil \frac{n}{k}\rceil}{n} M_i^{S(n,k,i)}}
\end{equation*}
\end{proposition}
\begin{proof}
The proof is left to the reader following examples shown on Figure  ~\ref{SLARCSCableMn1}. $S(n,k,i)$ is the sum of products $\displaystyle{\prod_{j=1}^i }\sbinom{k}{\lambda_j}$,  over all possible
partitions $\lambda=(\lambda_1, \ldots, \lambda_i)$ of $n$ into $i$ blocks of length at most $k$.
\end{proof}
\begin{figure}[h]
\begin{center}
\scalebox{1.2}{
\includegraphics{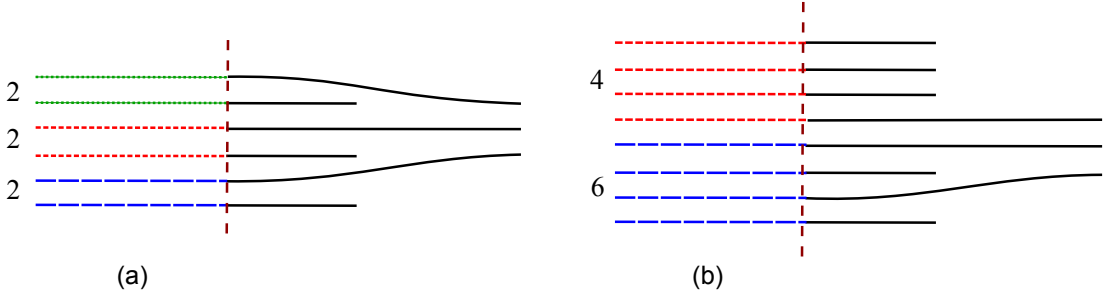}}
\caption{(a) $2$-cabling of $M_3$; (b) 4-cabling of $M_3$ corresponding to the partition
$(2,1)$: 2 arcs in the same part contribute 6 hence, the total
contribution is 24.} \label{SLARCSCableMn1}
\end{center}
\end{figure}

We compute cabling modules of $M_n$ for small values
of $n$: ${}^{[k]}M_0=M_0$, ${}^{[k]}M_1=M_1^{k}$,  ${}^{[k]}M_2=M_2^{k^2} \oplus M_1^{\sbinom{k}{2}}$,
${}^{[k]}M_3=M_3^{k^3}\oplus M_2^{2\sbinom{k}{1}\sbinom{k}{2}} \oplus M_1^{\sbinom{k}{3}}$.\\

Studying cablings of projective modules reduces to the case of standard modules:
 ${}^{[k]}P_n$ has a filtration with the $i$-th term consisting of $\sbinom{n}{i} {}^{[k]}M_i$,
based on the filtration \ref{SLArcsFilterPn} of $P_n$ by $P_n(i)$, $i \leq n.$

 Cabling functor ${}^{[k]}$ is exact, sending an $A^-$-module $M$ to its $k$-cabled module
${}^{[k]}M$, and categorifies the following operator on the Grothendieck
group:

$$[M_n]=(x-1)^n \mapsto [{}^{[k]}M_n]=\displaystyle{\sum_{i=\lceil \frac{n}{k}\rceil}^nS(n,k,i)(x-1)^i}.$$

Notice that ${}^{[s][k]}M \cong {}^{[ks]}M$ functorially in $M$.

\begin{prop}
  The cabling functor ${}^{[k]}$ preserves finitely generated $A^-$-modules.
\end{prop}
\begin{proof}
  ${}^{[k]}M_n$ is finitely generated. Indecomposable projective module $P_m$ has a finite filtration by standard modules (\ref{SLArcsFilterPn}),
 therefore ${}^{[k]}P_m$ is finitely generated. A finitely generated module $M$ is a quotient of finite sum of indecomposable projective modules $P_m$,
thus ${}^{[k]}M$ is finitely generated, and the functor ${}^{[k]}$ preserves the category $A^- \dmod.$

\end{proof}

Another cabling functor, denoted by $\mathfrak{L}_k$, on the category $A^- \pmod$ can be defined on objects by sending $\mathfrak{L}_k (P_n)=P_{nk}$
 and on morphisms in the same way as above, Figure \ref{SLARCSCableY}, i.e.  $\mathfrak{L}_k (\alpha)={}^{[k]}\alpha$ for $\alpha \in {}_mB_n^-.$

Given a full subcategory $\mathcal{A} \subset \mathcal{B},$ we say that endofunctors $F:\mathcal{A} \ra \mathcal{A} $ and $G:\mathcal{B} \ra \mathcal{B} $
are \emph{weakly adjoint} if
$$\Hom_{\mathcal{B}} (FM_1, M_2) \cong \Hom_{\mathcal{B}} (M_1, GM_2),$$
functorially in $M_1 \in \mathcal{A}$ and $M_2 \in \mathcal{B}.$

\begin{prop}
Cabling functors $\mathfrak{L}_k$ and ${}^{[k]}$ acting on categories  $A^- \pmod$ and $A^- \dmod$, respectively, are weakly adjoint.
\end{prop}
\begin{proof} It is sufficient to prove the statement for indecomposable projective modules $P_n \in A^- \pmod$ and any module $M \in A^- \dmod$.
  \begin{equation*}
  \Hom(\mathfrak{L}_k (P_n),M)\cong \Hom(P_{nk},M)\cong 1_{nk}M \cong 1_n{}^{[k]}M \cong \Hom(P_n,{}^{[k]}M).
 \end{equation*}

\end{proof}

\section{Monoidal structure}\label{MS}

Full subcategory $C^-$ of $A^- \pmod$ which consists of objects $P_n$, $n \geq 0$ is monoidal and preadditive,
with the unit object $\mathbf{1}=P_0$ and a single generating object $P_1$, since $P_n=(P_1)^{\otimes n}$.
One can think of $C^-$ as a monoidal category with generating object $P_1$, generating morphisms $a \in \Hom(P_1,P_0)$ and $b \in \Hom(P_0,P_1)$ and defining relation setting the value of the floating arc viewed  as an endomorphism of $\mathbf{1}$, to zero, see Figure \ref{SLARCSMonoidal}.
\begin{figure}[h]
\begin{center}
\scalebox{1.3}{
\includegraphics{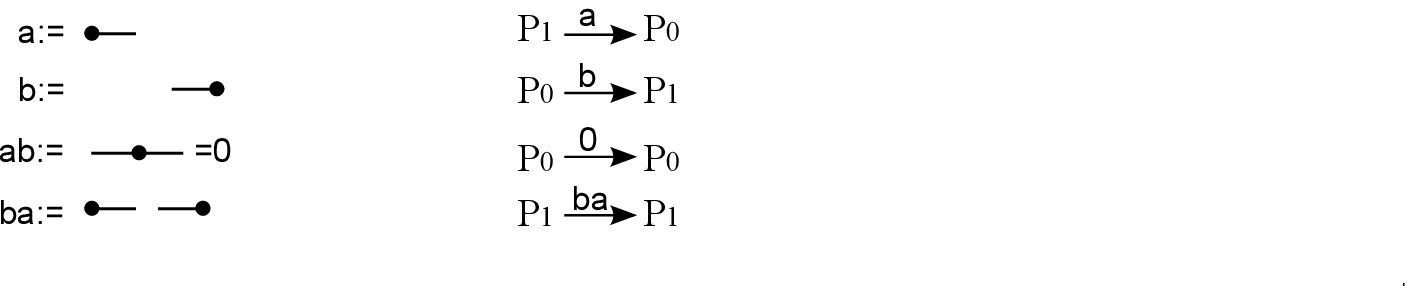}}
\caption{Generating morphisms in category $C^-$.} \label{SLARCSMonoidal}
\end{center}
\end{figure}

$C^-$ is a monoidal $\fieldk$-linear category such that:
\begin{eqnarray*}
  \Hom(P_0,\mathbf{1})&=&\fieldk \\
  \Hom(P_0, P_1) &=& \fieldk a \\
  \Hom(P_1, P_0) &=& \fieldk b \\
  \Hom(P_1,P_1) &=& \fieldk 1 \oplus \fieldk ab
\end{eqnarray*}

From this point of view, the SLarc algebra $A^-$ can be viewed as the $\Hom$ algebra of the monoidal category $C^-$:
 $$A^-=\oplusop{n,m \geq 0} \Hom(P_1^{\otimes n},P_1^{\otimes m}).$$

\begin{prop}\label{SLARCSMonStructMn}
  Standard module $M_n$ is isomorphic to the $n-$th derived tensor product of $M_1$: $M_n \simeq M_1^{\widehat{\otimes} n}.$
\end{prop}
\begin{proof}
  The minimal projective resolution of $M_1$ is
  \begin{equation}\label{SLARCSResM1}
    0 \ra P_0 \ra P_1 \ra 0
  \end{equation}
  The $n$-th derived tensor power $M_1^{\widehat{\otimes} n}$ can be computed by substituting this resolution
  for each term in the tensor product
   $ M_1^{\otimes n} \mapsto( 0 \ra P_0 \ra P_1 \ra 0)^{\otimes n}.$ This tensor power will contain $2^n$ terms of the form
   $$P_{\epsilon_1} \otimes P_{\epsilon_2} \otimes \ldots \otimes P_{\epsilon_n}=P_{\epsilon_1+\epsilon_2+\ldots +\epsilon_n}$$
  for $\epsilon_i \in \{0,1 \}$.

  Projective module $P_m$ will appear $\sbinom{n}{m}$ times in the complex, and it is easy to match the resulting complex to the   projective resolution (\ref{SLARCSResMnbyPn}) of the standard module $M_n$.
\end{proof}
Proposition \ref{SLARCSMonStructMn}, see also Corollary \ref{tensorMn}, generalizes the observation that $[M_n]=(x-1)^n=[M_1]^n.$\\

\section{A modification of $A^-$}

Assuming that we work over a field $\fieldk$, we have two canonical choices for the value of the floating arc: either $0$ or $1$.
Choosing value zero yields described categorification of the polynomial ring and, interestingly enough, value one leads to yet another
categorification of the polynomial ring. Let us denote by $A^+$ this modification of the SLarc algebra $A^-$. Elements $1_n$ and
projective modules $P_n$ are defined as in $A^-$ algebra case.

\begin{figure}[h]
\scalebox{.99}{\includegraphics{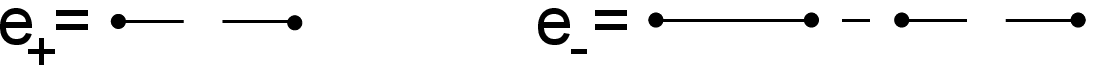}}
 \caption{Idempotents $e_+$ and $e_-$ in $_1A^+_1.$}
 \label{SLARCid11}
\end{figure}

However, changing the value of the floating arc from $0$ to $1$ produces additional idempotents, such as element
$e_+ \in {}_1B_1^+$ which is an idempotent according to calculation
shown on Figure \ref{AplusId}, and the complementary idempotent $e_-=1_1 -e_+$, see Figure \ref{SLARCid11}.

\begin{figure}[h]
\scalebox{.99}{\includegraphics{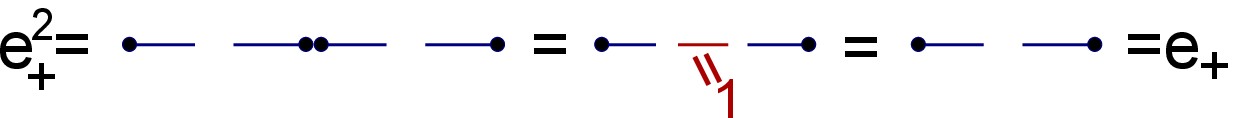}}
 \caption{Element $e_+$ is an idempotent in algebra $A^-$.}
 \label{AplusId}
\end{figure}

Idempotents in $End(P_n)$ for any $n>1$ can be obtained from $e_+$ and  $e_-$
by using the monoidal structure of $A^+ \pmod$ analogous to the one in $A^- \pmod$,
for which  $P_n \otimes P_m=P_{n+m}. $

Let $\varepsilon=(\varepsilon_1,\varepsilon_2, \ldots, \varepsilon_n)$, $\varepsilon_i \in \{+,-\},$
denote a sequence of pluses and minuses of length $n$, and $(-^n)$ the sequence containing exactly $n$ minuses.
The corresponding idempotents are denoted by $e_{\varepsilon}$ and $e_{(-^n)}$, respectively. The natural tensor
product structure on $A^+ \pmod$  satisfies
$P_{\varepsilon} \otimes P_{\varepsilon'}=P_{\varepsilon \varepsilon'}.$
Idempotent $e_{\varepsilon}= \displaystyle{\otimes_{i=1}^n}e_{\varepsilon_i}$ is just a tensor product
of idempotents $e_+$ and  $e_-$'s, according to the sequence $\varepsilon$, for example see Figure  \ref{AplusMoreId}.

\begin{figure}[h]
\scalebox{.99}{\includegraphics{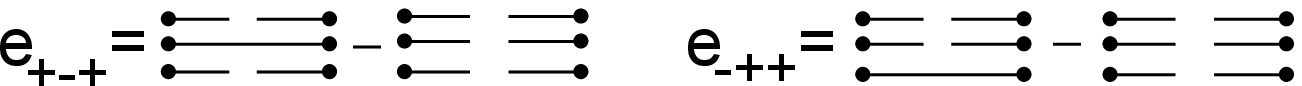}}
 \caption{Additional idempotents in algebra $_3A_3^+.$}
 \label{AplusMoreId}
\end{figure}

Notice that $1_n=\displaystyle{\sum_{|\varepsilon|=n}e_{\varepsilon}}.$ Moreover, these
idempotents are mutually orthogonal,
$e_{\varepsilon}e_{\varepsilon'}=\delta_{\varepsilon, \varepsilon'}e_{\varepsilon}.$
In particular, $e_+ e_-=e_-e_+=0.$

In general, given a ring $R$ and two idempotents $e,f \in R$, projective modules
$Re$ and $Rf$ are isomorphic iff there exist elements $a=d_{e \ra f}, b=d_{f \ra e} \in R$ such that
$eafbe=e$ and $fbeaf=f.$ Moreover, in this case,  we say that the elements $e,f$
are equivalent, and denote that by $e \simeq f.$

\begin{lemma}\label{SLARC+lemma}
 If sequence $\varepsilon$ contains exactly $m$ minuses then  $e_{\varepsilon} \simeq e_{(-^m)}.$
\end{lemma}
\begin{proof}
  The equivalence is realized by maps corresponding to the following diagrams:
$d_{\varepsilon \ra m}$ with $n$ left and $m$ right endpoints and $m$ through arcs connecting right endpoints to
those left endpoints corresponding to the minus signs in $\varepsilon$, and the remaining points extended to short left arcs.
$b=d_{m \ra \varepsilon}$ is a reflection of $a=d_{\varepsilon \ra m}$ along the vertical axis. We have

\begin{eqnarray*}
  e_{(-^m)} \,d_{m \ra \varepsilon}\,e_{\varepsilon}\,d_{\varepsilon \ra m}\,e_{(-^m)}&=&e_{(-^m)} {\rm \,\,\,and \,}\\
  e_{\varepsilon} \,d_{\varepsilon \ra m}\,e_{(-^m)} \,d_{m \ra \varepsilon}\,e_{\varepsilon}&=&e_{\varepsilon}.
\end{eqnarray*}
An example is shown on Figure  \ref{AplusEEm}.

\begin{figure}[h]
\scalebox{1.2}{\includegraphics{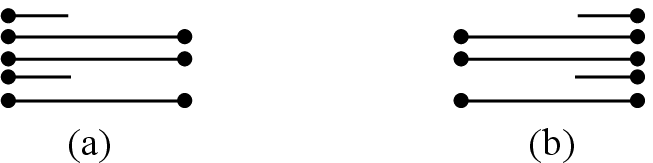}}
 \caption{Maps $d_{(-,+,-,-,+ )\ra(-^3)}$ and $d_{(-^3) \ra (-,+,-,-,+)}$.}
 \label{AplusEEm}
\end{figure}
\vspace{-0.5cm}\end{proof}

\begin{lemma}\label{SLARC+EPSEPS}
  If sequences $\varepsilon$ and $\varepsilon'$ contain $n$ and $m$ minuses, respectively,  then
 $e_{\varepsilon} \simeq e_{\varepsilon'}$ iff $m=n.$
\end{lemma}
\begin{proof}
 Based on Lemma \ref{SLARC+lemma} $e_{\varepsilon} \simeq e_{(-^n)}$ and  $e_{\varepsilon'} \simeq e_{(-^m)}$   and
$e_{(-^n)}$, $e_{(-^m)}$ are not equivalent unless $m=n$.
\end{proof}

\begin{cor}\label{SLARCS+ProjEq}
  Projective modules $A^+e_{\varepsilon}$ and $A^+e_{\varepsilon'}$ are isomorphic iff sequences $\varepsilon$ and
  $\varepsilon'$ contain the same number of minuses.
\end{cor}

To a sequence $(-^n)$ we assign an indecomposable projective $A^+$ module $P_{(-^n)}=A^+e_{(-^n)}.$

\begin{prop}\label{SLARCA+}
  Projective modules $P_{(-^n)}$ are simple objects satisfying the following properties:
\begin{enumerate}
\item $ \Hom (P_{(-^m)}, P_{(-^n)}) = \left\{
                                       \begin{array}{ll}
                                         \fieldk, & \hbox{n=m;} \\
                                         0, & \hbox{{\rm else.}}
                                       \end{array}
                                     \right.$

\item $ P_n \cong \oplusop{|\varepsilon|=n}P_{\varepsilon} \cong \oplusoop{m=0}{n} \sbinom{n}{m}P_{(-^m)}.$
   \end{enumerate}
\end{prop}
\begin{proof}\hspace{-0.5cm}
  \begin{enumerate}
    \item Follows from Proposition \ref{SLARCS+ProjEq} since
$$\Hom (P_{(-^m)}, P_{(-^n)})=\Hom (A^+e_{(-^m)}, A^+e_{(-^n)})=e_{(-^m)}A^+e_{(-^n)}.$$
    \item $ P_n=A^+1_n=\oplusop{|\varepsilon|=n}A^+e_{\varepsilon}=\oplusop{|\varepsilon|=n}P_{\varepsilon}.$ Each $P_{\varepsilon}$
is equivalent to $P_{(-^m)}$ and there are $\sbinom{n}{m}$ sequences $\varepsilon$ of length $n$ with exactly $m$ minuses.
     \end{enumerate}
\vspace{-0.5cm}\end{proof}

We see that the category $A^- \pmod$ of projective $A^-$-modules is semisimple. Idempotented ring $A^-$ is therefore semisimple and Morita equivalent to
idempotented ring $\fieldk \oplus \fieldk \oplus \ldots \oplus \fieldk $ which is a countable sum of copies of the field $\fieldk.$
Let $K_0(A^{+})$ denote the Grothendieck ring of a monoidal category of finitely-generated $A^{+}$ projective modules.
As before, $[P_0]=1$ and $[P_1]=x$, $[P_n]=x^n.$ Based on the decomposition of the projective modules in Proposition \ref{SLARCA+}(2)
we conclude that $[P_{(-^n)}]=(x-1)^n.$


\end{document}